\documentclass{daj}

\usepackage{amscd}
\usepackage{amsfonts}
\usepackage{amsmath}
\usepackage{amssymb}
\usepackage{amsthm}
\usepackage{enumitem}
\usepackage{float}
\usepackage{mathtools}
\usepackage{mathrsfs}
\usepackage{needspace}
\usepackage{setspace}
\usepackage{tikz}
\usepackage{version}

\usepackage{hyperref}

\usetikzlibrary{calc,positioning}

\theoremstyle{plain}
\newtheorem{lemma}{Lemma}[section]
  \newtheorem{proposition}[lemma]{Proposition}
  
  \newtheorem{theorem}[lemma]{Theorem}
  \newtheorem*{theorem*}{Theorem}

  \newtheorem*{fact*}{Fact}
  \newtheorem*{claim*}{Claim}

  \newtheorem{question}[lemma]{Question}
  
  \newtheorem{slogan}[lemma]{Slogan}

\theoremstyle{definition}
  \newtheorem{definition}[lemma]{Definition}

\theoremstyle{remark}
  \newtheorem{remark}[lemma]{Remark}

\setenumerate{label=\textup{(}\roman*\textup{)}}

\makeatletter
\newsavebox{\@brx}
\newcommand{\llangle}[1][]{\savebox{\@brx}{\(\m@th{#1\langle}\)}%
  \mathopen{\copy\@brx\kern-0.5\wd\@brx\usebox{\@brx}}}
\newcommand{\rrangle}[1][]{\savebox{\@brx}{\(\m@th{#1\rangle}\)}%
  \mathclose{\copy\@brx\kern-0.5\wd\@brx\usebox{\@brx}}}
\makeatother
\DeclareMathOperator*{\EEE}{\EE}

\newcommand\imCMsym[4][\mathord]{%
  \DeclareFontFamily{U} {#2}{}
  \DeclareFontShape{U}{#2}{m}{n}{
    <-6> #25
    <6-7> #26
    <7-8> #27
    <8-9> #28
    <9-10> #29
    <10-12> #210
    <12-> #212}{}
  \DeclareSymbolFont{CM#2} {U} {#2}{m}{n}
  \DeclareMathSymbol{#4}{#1}{CM#2}{#3}
}

\imCMsym{cmmi}{124}{\CMjmath}
\imCMsym[\mathop]{cmsy}{113}{\CMamalg}
\imCMsym[\mathop]{cmex}{96}{\CMcoprod}

\dajAUTHORdetails{%
  title = {Good Bounds in Certain Systems of True Complexity One},
  author = {Freddie Manners},
  plaintextauthor = {Freddie Manners},
}   %%% END \dajAUTHORdetails

\dajEDITORdetails{%
   year={2018},
   number={21},
   received={27 September 2017},   % received date: example: 7 January 2017
   published={28 December 2018},  % published date
   doi={10.19086/da.6814},       % XXX = number of paper, e.g. da006 for paper#6
}   %%% END \dajEDITORdetails

\begin{document}

\newcommand{\eps}[0]{\varepsilon}

\newcommand{\AAA}[0]{\mathbb{A}}
\newcommand{\CC}[0]{\mathbb{C}}
\newcommand{\EE}[0]{\mathbb{E}}
\newcommand{\FF}[0]{\mathbb{F}}
\newcommand{\NN}[0]{\mathbb{N}}
\newcommand{\PP}[0]{\mathbb{P}}
\newcommand{\QQ}[0]{\mathbb{Q}}
\newcommand{\RR}[0]{\mathbb{R}}
\newcommand{\TT}[0]{\mathbb{T}}
\newcommand{\ZZ}[0]{\mathbb{Z}}

\newcommand{\cA}[0]{\mathcal{A}}
\newcommand{\cB}[0]{\mathcal{B}}
\newcommand{\cC}[0]{\mathcal{C}}
\newcommand{\cD}[0]{\mathcal{D}}
\newcommand{\cE}[0]{\mathcal{E}}
\newcommand{\cF}[0]{\mathcal{F}}
\newcommand{\cH}[0]{\mathcal{H}}
\newcommand{\cG}[0]{\mathcal{G}}
\newcommand{\cK}[0]{\mathcal{K}}
\newcommand{\cM}[0]{\mathcal{M}}
\newcommand{\cN}[0]{\mathcal{N}}
\newcommand{\cP}[0]{\mathcal{P}}
\newcommand{\cR}[0]{\mathcal{S}}
\newcommand{\cS}[0]{\mathcal{S}}
\newcommand{\cT}[0]{\mathcal{S}}
\newcommand{\cU}[0]{\mathcal{U}}
\newcommand{\cV}[0]{\mathcal{V}}
\newcommand{\cW}[0]{\mathcal{W}}
\newcommand{\cX}[0]{\mathcal{X}}
\newcommand{\cY}[0]{\mathcal{Y}}
\newcommand{\cZ}[0]{\mathcal{Z}}

\newcommand{\fg}[0]{\mathfrak{g}}
\newcommand{\fs}[0]{\mathfrak{s}}
\newcommand{\fk}[0]{\mathfrak{k}}
\newcommand{\fZ}[0]{\mathfrak{Z}}

\newcommand{\bmu}[0]{\boldsymbol\mu}
\newcommand{\AUT}[0]{\mathbf{Aut}}
\newcommand{\Aut}[0]{\operatorname{Aut}}
\newcommand{\Frob}[0]{\operatorname{Frob}}
\newcommand{\GI}[0]{\operatorname{GI}}
\newcommand{\HK}[0]{\operatorname{HK}}
\newcommand{\HOM}[0]{\mathbf{Hom}}
\newcommand{\Hom}[0]{\operatorname{Hom}}
\newcommand{\Ind}[0]{\operatorname{Ind}}
\newcommand{\Lip}[0]{\operatorname{Lip}}
\newcommand{\LHS}[0]{\operatorname{LHS}}
\newcommand{\RHS}[0]{\operatorname{RHS}}
\newcommand{\Sub}[0]{\operatorname{Sub}}
\newcommand{\id}[0]{\operatorname{id}}
\newcommand{\image}[0]{\operatorname{Im}}
\newcommand{\poly}[0]{\operatorname{poly}}
\newcommand{\trace}[0]{\operatorname{Tr}}
\newcommand{\sig}[0]{\ensuremath{\tilde{\cS}}}
\newcommand{\psig}[0]{\ensuremath{\cP\tilde{\cS}}}
\newcommand{\metap}[0]{\operatorname{Mp}}
\newcommand{\symp}[0]{\operatorname{Sp}}
\newcommand{\dist}[0]{\operatorname{dist}}
\newcommand{\stab}[0]{\operatorname{Stab}}
\newcommand{\HCF}[0]{\operatorname{hcf}}
\newcommand{\LCM}[0]{\operatorname{lcm}}
\newcommand{\SL}[0]{\operatorname{SL}}
\newcommand{\GL}[0]{\operatorname{GL}}
\newcommand{\rk}[0]{\operatorname{rk}}
\newcommand{\sgn}[0]{\operatorname{sgn}}
\newcommand{\uag}[0]{\operatorname{UAG}}
\newcommand{\freiman}[0]{Fre\u{\i}man}
\newcommand{\tf}[0]{\operatorname{tf}}
\newcommand{\ev}[0]{\operatorname{ev}}
\newcommand{\Conv}[0]{\mathop{\scalebox{1.5}{\raisebox{-0.2ex}{$\ast$}}}}
\newcommand{\bs}[0]{\backslash}
\newcommand{\heis}[3]{ \left(\begin{smallmatrix} 1 & \hfill #1 & \hfill #3 \\ 0 & \hfill 1 & \hfill #2 \\ 0 & \hfill 0 & \hfill 1 \end{smallmatrix}\right)  }
\newcommand{\uppar}[1]{\textup{(}#1\textup{)}}
\newcommand{\pol}[0]{\l}
\newcommand{\sslash}{\mathbin{/\mkern-6mu/}}
\newcommand{\PGL}[0]{\operatorname{PGL}}
\newcommand{\PSL}[0]{\operatorname{PSL}}
\newcommand{\spn}[0]{\operatorname{span}}
\newcommand{\im}[0]{\operatorname{im}}
\newcommand{\CS}[0]{\operatorname{CS}}
\newcommand{\MERGE}[0]{\operatorname{MERGE}}
\newcommand{\TRIVIAL}[0]{\operatorname{TRIVIAL}}

\begin{frontmatter}[classification=text]
%% EDITOR: this will force the keywords to appear right after the Abstract.
%%   If the abstract is too long and would force the keywords off the
%%   front page, please comment out % [classification=text] above
%%   This way the keywords will be floated on the bottom of the first page
%%   even though the Abstract spills over to the next page.

%%% AUTHOR: Title goes here.  This line is optional.  You must use it
%%   if title has footnote attached or requires nontrivial typesetting,
%%   e.g., inclusion of linebreaks to force nice layout.
%\title{Short Proof of R\"odl's $n^{\log\log n}$ Bound\footnote{This is a footnote to the title}} %% please capitalize all significant words

%%% AUTHOR:
%%% List all authors. If you wish, place grant acknowledgements in \thanks.
%%% In brackets include a short tag for each author.
\author[f]{Freddie Manners}

%%% AUTHOR: Abstract goes here
\begin{abstract}
  Let $\Phi = (\phi_1,\dots,\phi_6)$ be a system of $6$ linear forms in $3$ variables, i.e.~$\phi_i \colon \ZZ^3 \to \ZZ$ for each $i$.  Suppose also that $\Phi$ has Cauchy--Schwarz complexity $2$ and true complexity $1$, in the sense defined by Gowers and Wolf; in fact this is true generically in this setting.  Finally let $G = \FF_p^n$ for any $p$ prime and $n \ge 1$.  Then we show that multilinear averages by $\Phi$ are controlled by the $U^2$-norm, with a polynomial dependence; i.e.~if $f_1,\dots,f_6 \colon G \to \CC$ are functions with $\|f_i\|_{\infty} \le 1$ for each $i$, then for each $j$, $1 \le j \le 6$:
  \[
    \left| \EEE_{x_1,x_2,x_3 \in G} f_1(\phi_1(x_1,x_2,x_3)) \dots f_6(\phi_6(x_1,x_2,x_3)) \right| \le \|f_j\|_{U^2}^{1/C}
  \]
  for some $C > 0$ depending on $\Phi$.  This recovers and strengthens a result of Gowers and Wolf in these cases.  Moreover, the proof uses only multiple applications of the Cauchy--Schwarz inequality, avoiding appeals to the inverse theory of the Gowers norms.

  We also show that some dependence of $C$ on $\Phi$ is necessary; that is, the constant $C$ can unavoidably become large as the coefficients of $\Phi$ grow.
\end{abstract}
\end{frontmatter}

%%% AUTHOR: body of paper starts here

\section{Introduction}

Let $G$ be a finite abelian group and $X \subseteq G$ a subset.  Many problems of interest in additive combinatorics and related fields involve counting solutions to some system of equations within $X$.  For instance, we might wish to count Schur triples $(x,y,x+y)$ all of whose coordinates lie in $X$, or $4$-term arithmetic progressions $(x, x+h, x+2h, x+3h)$ where each term lies in $X$, etc..

The most general case of this kind of question is as follows: given a tuple of $r$ linear forms $\Phi = (\phi_1,\dots,\phi_r)$ where $\phi_i \colon \ZZ^d \to \ZZ$, and $r$ functions $f_1,\dots,f_r \colon G \to \CC$, estimate the quantity
\[
  \Lambda_\Phi(f_1,\dots,f_r) = \EEE_{x_1,\dots,x_d \in G} f_1(\phi_1(x_1,\dots,x_d)) \dots f_r(\phi_r(x_1,\dots,x_d)) \, .
\]
(Here we have abused notation to let $\phi_i \colon \ZZ^d \to \ZZ$ induce a function $G^d \to G$, in the obvious way.)

So, in our examples above, we take $f_1 = \dots = f_r = 1_X$, the indicator function of $X$; however it is convenient to allow more general functions in the definition of $\Lambda_\Phi$ as they arise in intermediate computations.  In our examples, $\Phi$ is as follows:
\begin{itemize}
  \item in the case of Schur triples, $r=3$, $d=2$, $\phi_1(x,y) = x$, $\phi_2(x,y) = y$ and $\phi_3(x,y) = x+y$;
  \item in the case of $4$-term arithmetic progressions, $r=4$, $d=2$, $\phi_1(x,y) = x$, $\phi_2(x,y) = x+y$, $\phi_3(x,y) = x+2y$ and $\phi_4(x,y) = x + 3y$.
\end{itemize}

A fundamental observation in much recent progress in such questions (as applied to Szemer\'edi-type theorems, or counting solutions to linear equations in the primes), originally due to Gowers \cite{gowers-aps}, is that averages $\Lambda_\Phi$ are controlled by Gowers uniformity norms.\footnote{We will assume the reader is familiar with the definition of Gowers norms and some related concepts; for an introduction, see e.g.~\cite[Appendix B]{green-tao-primes}, \cite[Chapter 11]{tao-vu}, or \cite{tao-book}.}

A weak statement of this type is that if $X$ has density $\alpha$ and $X$ is suitably quasirandom in the sense that $\|1_X - \alpha\|_{U^{s+1}} = o(1)$, where $s$ is some positive integer, then
\begin{equation}
  \label{eq:weak-gowers-control}
  \Lambda_\Phi(1_X,\dots,1_X) = \alpha^r + o(1)
\end{equation}
i.e.~the number of solutions to the system $\Phi$ in $X$ is roughly the same as the expected count in a random set, i.e.~$\alpha^r |G|^d$.  A stronger type of statement one could make is that if $f_1,\dots,f_d \colon G \to \CC$ are any functions with $\|f_i\|_\infty \le 1$ for all $i$, and $\|f_j\|_{U^{s+1}} = o(1)$ for any one $j \in \{1,\dots,d\}$, then
\begin{equation}
  \label{eq:strong-gowers-control}
  \left|\Lambda_\Phi(f_1,\dots,f_r)\right| = o(1) \, ;
\end{equation}
and indeed this kind of statement implies the previous one. 

The remaining question is when one has such a statement for a system of linear forms $\Phi$, and if so, how small the positive integer $s$ can be; i.e.~how far one has to go in the hierarchy of Gowers norms to control $\Lambda_\Phi$.  For instance, for $k$-term arithmetic progressions, Gowers \cite{gowers-aps} showed that a statement of type \eqref{eq:strong-gowers-control} holds for $s = k - 2$, and with a good bound; specifically,
\[
  \left|\Lambda_{k\mathrm{AP}}(f_1,\dots,f_k)\right\| \le \|f_j\|_{U^{k-1}}
\]
whenever $\|f_i\|_\infty \le 1$ for each $i$, and for any $1 \le j \le k$. The proof is $k-1$ applications of the Cauchy--Schwarz inequality.

Moreover, Gowers gave examples to show that $s=k-2$ cannot be improved.  For instance, when $k=4$ and $G = \ZZ/p\ZZ$, we can consider functions
\begin{align*}
  f_1(x) &= \exp(2 \pi i a x^2 / p) &
  f_3(x) &= \exp(2 \pi i a x^2 / p) \\
  f_2(x) &= \exp(-2 \pi i a x^2 / p) &
  f_4(x) &= \exp(-2 \pi i a x^2 / p)
\end{align*}
for some nonzero $a \in \ZZ/p\ZZ$, and observe that $f_1(x) f_2(x+h) f_3(x+2h) f_4(x+3h) = 1$ pointwise for any $x,h \in G$.  So, $\Lambda_{4\mathrm{AP}}(f_1,\dots,f_4) = 1$, but one can show that $\|f_j\|_{U^2} = O(p^{-1/2})$.  This rules out a statement of type \eqref{eq:strong-gowers-control} for $s=1$; taking appropriate level sets of these functions rules out \eqref{eq:weak-gowers-control} also.

The first systematic approach to this question for general systems of linear forms was given by Green and Tao \cite{green-tao-primes} in the course of their work on linear equations in primes.  The following is essentially implicit as a much easier case of results from that paper, and was isolated in \cite{gowers-wolf-1}; however, the terminology we use is slightly different to both.
\begin{proposition}[Essentially from {\cite{green-tao-primes}}]
  \label{prop:csc}
  Given a prime $p$, a system $\Phi = (\phi_1,\dots,\phi_r)$ of linear forms $\ZZ^d \to \ZZ$, and an index $j$, $1 \le j \le r$, we say $\Phi$ has \emph{Cauchy--Schwarz complexity $\le{}s$ at $j$, modulo $p$}, if the following holds: the indices $\{1,\dots,r\} \setminus \{j\}$ can be partitioned into $s+1$ classes $I_1,\dots,I_{s+1}$ such that $\phi_j$ modulo $p$, considered as a linear form $\FF_p^d \to \FF_p$, is not contained in $\spn_{\FF_p}(\phi_i \colon i \in I_k)$ for any $1 \le k \le s+1$.

  Let $G = \FF_p^n$, where $p$, $n$ may be any size \uppar{including say $n=1$}.
  If $\Phi$ has Cauchy--Schwarz complexity $\le s$ at $j$ modulo $p$, then for any functions $f_1,\dots,f_d \colon G \to \CC$ with $\|f_i\|_\infty \le 1$ for each $i$, we have
  \[
    \left| \Lambda_\Phi(f_1,\dots,f_r) \right| \le \|f_j\|_{U^{s+1}} \, .
  \]
\end{proposition}
If $\Phi$ has Cauchy--Schwarz complexity $\le s$ at every index, modulo $p$, we could just say the system has Cauchy--Schwarz complexity $\le s$ modulo $p$, and write $s_{\CS}(\Phi)$ for the smallest $s$ for which this holds (where $s_{\CS}$ implicitly depends on $p$).  If $p$ is very large, taking the span of $\phi_1,\dots,\phi_r$ as linear forms $\FF_p^d \to \FF_p$ is essentially equivalent to working over $\QQ$ and the value of $s_{\CS}$ stabilizes.

As the name suggests, the proof of Proposition \ref{prop:csc} is $s+1$ applications of Cauchy--Schwarz, as in Gowers' work.  The content of the proposition is really in establishing the linear algebra condition that guarantees this Cauchy--Schwarz argument will work.

Following this, Gowers and Wolf, in a series of papers \cite{gowers-wolf-1,gowers-wolf-2,gowers-wolf-3,gowers-wolf-4}, considered the question: is the value of $s$ given by Cauchy--Schwarz complexity optimal?  It is natural to try to adapt the examples given by Gowers for $k$-term progressions to the general case to give a lower bound.  The task comes down to finding phase polynomials $f_1,\dots,f_r \colon G \to \CC$ of degree $D$, i.e.~functions of the form $f_j(x) = \exp(2 \pi i P_j(x))$ where $P_j \colon G \to \RR/\ZZ$ is a degree $D$ polynomial (in a natural sense) for each $j$, such that
\[
  \Lambda_\Phi(f_1,\dots,f_r) = 1
\]
i.e.~the multilinear average is equal to $1$ pointwise.  By contrast $\|f_j\|_{U^{D}}$ will typically be very small when $f_j$ is a degree $D$ phase polynomial, so this rules out a statement of type \eqref{eq:strong-gowers-control} or \eqref{eq:weak-gowers-control} for $s\le D-1$.

It turns out that this is possible if and only if $\phi_1^D,\dots,\phi_r^D \in ((\FF_p^d)^{\ast})^{\otimes D}$ are linearly dependent, where $\phi_i^D = \phi_i^{\otimes D}$ are interpreted as symmetric multilinear forms over $\FF_p^d$.  In other words, $\Lambda_\Phi$ can only be fully controlled by the $U^{s+1}$-norm if $\phi_1^{s+1},\dots,\phi_r^{s+1}$ are linearly independent elements of $((\FF_p^d)^{\ast})^{\otimes D}$.

It also turns out that this lower bound, arising from explicit phase polynomials, and the upper bound coming from Cauchy--Schwarz complexity, do not agree in general.  Gowers and Wolf conjectured that the lower bound is the truth; that is, if the \emph{true complexity} of $\Phi$ over $\FF_p$ is defined to be the smallest $s$ such that $\phi_1^{s+1},\dots,\phi_r^{s+1}$ are linearly independent,\footnote{In fact Gowers and Wolf set up the definitions slightly differently: they define true complexity to be the smallest $s$ such that \eqref{eq:weak-gowers-control} holds, and conjecture it is equal to the algebraic quantity we have just defined.  Since this conjecture is known to be true in cases of interest, defining things the other way round should hopefully not cause too much confusion.} then \eqref{eq:strong-gowers-control} holds for this $s$ and any $1 \le j \le r$.  By our previous discussion, such a statement would be (qualitatively) best possible in $s$.

In what follows, we write $s = s(\Phi)$ to denote this notion of the true complexity of a system of linear forms (over $\FF_p$), unless otherwise stated.

This conjecture has now been resolved in essentially all cases of interest. The original paper \cite{gowers-wolf-1} by Gowers and Wolf proved the case where $s=1$, $s_{\CS} = 2$, and $G = \FF_p^n$ for $p$ fixed and $n$ large.  This case was proved again in \cite{gowers-wolf-3} (also by Gowers and Wolf) with an improved quantitative bound on $|\Lambda_{\Phi}(f_1,\dots,f_r)|$ in terms of $\|f_j\|_{U^{s+1}}$.  Still when $G = \FF_p^n$ for $p$ fixed, but not too small, the general case (i.e.~arbitrary finite $s$ and $s_{\CS}$) was proven in another paper \cite{gowers-wolf-2} by the same authors.  They also showed the case $s_{\CS} = 2$ and $s=1$ for $G = \ZZ/p\ZZ$, where this time $p$ is large, in \cite{gowers-wolf-4}.  The general result in the cyclic setting $G = \ZZ/p\ZZ$ for $p$ large was shown by Green and Tao \cite{green-tao-arithmetic-regularity} as an application of their nilsequence-based arithmetic regularity lemma.

Later, Hatami and Lovett \cite{hatami-lovett} extended the results of \cite{gowers-wolf-2} to the asymmetric case, where $\phi_1^{s+1},\dots,\phi_r^{s+1}$ may be linearly dependent but not all of these multilinear forms are in the linear span of the others, which corresponds to \eqref{eq:strong-gowers-control} holding for some choices of $j$ but not others. Finally, Hatami, Hatami and Lovett \cite{hatami-hatami-lovett} removed the requirement in the case $G = \FF_p^n$ for $p$ fixed that $p$ be not too small.

We comment only very briefly on the proofs, as they will not play a large role in the current work.  We focus on the simplest case of $p$ fixed, $s=1$ and $s_{\CS} = 2$.  By the assumption on $s_{\CS}$ and Proposition \ref{prop:csc}, we are free to discard small $U^{3}$ errors at any point.  By the inverse theorem for the Gowers $U^3$-norm, this means we are free to assume that each $f_i$ is a linear combination of a few phase polynomials of degree at most $2$.  We would like to argue that when $s=1$, the quadratic terms do not contribute much to $\Lambda_\Phi$; however, this requires a more robust version of our assumption that $\phi_1^2,\dots,\phi_r^2$ are linearly independent (in effect, we need that no non-trivial linear combination over $((\FF_p^{nd})^{\ast})^{\otimes 2}$ has low rank).  Bridging this gap between the robust and non-robust statements is the heart of the argument.

At least qualitatively, these works verify all the central conjectures concerning true complexity.  Nonetheless, there are some unresolved questions of interest.
\begin{question}
  \label{q:bounds}
  What are the best possible bounds in the true complexity statement?  That is, how small must $\delta$ be in terms of $\eps$ to ensure that $\|f_j\|_{U^{s+1}} \le \delta$ implies $|\Lambda_{\Phi}(f_1,\dots,f_d)| \le \eps$?
\end{question}
In the case $s=1$ and $s_{\CS}=2$, Gowers and Wolf \cite{gowers-wolf-3, gowers-wolf-4} obtained a doubly exponential dependence, i.e.~$\delta \approx \exp \exp(-O(\eps^{-C}))$.\footnote{One of these exponentials can probably be removed using subsequent improved bounds in the inverse theorem for the $U^3$-norm that follow from work of Sanders \cite{sanders}.  The author is grateful to the anonymous reviewer for pointing this out.}  In all other cases where $s \ne s_{\CS}$ the best known bounds are ineffective, or as good as ineffective, as they rely on the inverse theorems for the $U^k$-norms for $k \ge 4$ for which no good bounds are known.

In \cite[Problem 7.8]{gowers-wolf-2}, Gowers and Wolf suggested that the dependence cannot be too good, and specifically, not polynomial; that is, they asked whether one could find a counterexample ruling out $\delta \approx \eps^C$.

This is closely related to the following question.

\begin{question}
  \label{q:elementary}
  In cases where the true complexity and Cauchy--Schwarz complexity differ, could a true complexity bound be proven by elementary means, e.g.~by many applications of the Cauchy--Schwarz inequality; or is some appeal to the structural theory of higher order Fourier analysis essential?  Is there some qualitative feature which separates the elementary and non-elementary cases?
\end{question}

The primary motivation behind Gowers and Wolf's appeal for counterexamples to good bounds is that this would rule out a proof based only on complicated applications of the Cauchy--Schwarz inequality, as that would surely give a polynomial bound.

Our final question is at first appearance more eccentric but we will see its relevance shortly.

\begin{question}
  \label{q:coefficients}
  Working over $G= \ZZ/p\ZZ$ for $p$ a large prime, and when $s \ne s_{\CS}$, all the known bounds on $\eps$ in terms of $\delta$ depend on the coefficients of the linear forms in $\Phi$, not just on the values of $s$ and $s_{\CS}$.  In practice, these results are only effective if the coefficients are essentially bounded.  

  By contrast, the Cauchy--Schwarz complexity bound is completely uniform in the coefficients, provided the hypotheses are satisfied.

  Is this restriction to bounded coefficients necessary, whenever $s \ne s_{\CS}$?

  Working over $\FF_p^n$ for $p$ fixed and $n$ large, there are only finitely many choices of linear forms, so any dependence on the coefficients can be removed. In this setting, the analogous question is whether the bounds should genuinely depend on $p$.
\end{question}

In this paper, we consider what is in some sense the smallest non-trivial case where $s \ne s_{\CS}$, which concerns systems of $6$ linear forms in $3$ variables (i.e., $r=6$ and $d=3$).  Indeed, when $d=2$ it is always the case that $s = s_{\CS}$, and similarly for $d=3$ and $r \le 5$.  However, a generic system $\Phi$ with $r=6$ and $d=3$ will have $s=1$ but $s_{\CS}=2$ (see Section \ref{sec:63} for a discussion).

In this limited setting, we are able to give fairly complete answers to the questions above.  We now outline the main results.

\begin{theorem}
  \label{thm:positive}
  Let $\Phi = (\phi_1,\dots,\phi_6)$ be a system of $6$ linear forms in $3$ variables, let $p$ be a prime \uppar{not necessarily small}, and let $G = \FF_p^n$ for any $n \ge 1$.

  Then, provided the system $\Phi$ has true complexity $1$ over $\FF_p$, for any functions $f_1,\dots,f_6 \colon G \to \CC$ with $\|f_i\|_\infty \le 1$ for each $i$, and for any $j$, $1 \le j \le 6$, we have the bound
  \[
    \left| \Lambda_\Phi(f_1,\dots,f_6) \right| \le \|f_j\|_{U^2(G)}^{1/C}
  \]
where $C = C(\Phi, j) > 0$ is some constant depending on the coefficients of $\Phi$, and perhaps $j$, but crucially not on $p$ or $n$.

  Moreover, the above inequality can be derived only using multiple applications of the Cauchy--Schwarz inequality. However, the number of applications used in the proof increases without bound as the coefficients of $\Phi$ grow.
\end{theorem}

Note that, since no restrictions are placed on $n$ and $p$, this encompasses the cases $\ZZ/p\ZZ$ for $p$ a large prime as well as $\FF_p^n$ for $p$ fixed and $n$ large. In intermediate cases where $p$ and $n$ are both large, even qualitatively the result may officially be new, although these cases are rarely of interest.

The key observation underlying the proof is the following.
\begin{slogan}
  \label{slogan:cs}
  Cauchy--Schwarz complexity is not preserved under applying the Cauchy--Schwarz inequality.
\end{slogan}
By this we mean the following.  If we start with a system of linear forms $\Phi$ and apply the Cauchy--Schwarz inequality to one of the functions, what we get can be thought of as a new linear system $\Phi'$ with $2(r-1)$ forms and $2d-1$ variables.  It is not always true that if $s_{\CS}(\Phi) > 1$ then $s_{\CS}(\Phi') > 1$; so in some cases we can now apply the Cauchy--Schwarz complexity bound (Proposition \ref{prop:csc}) to $\Lambda_{\Phi'}$ to bound it by $\|f_j\|_{U^2}$, meaning that in turn $\Lambda_{\Phi}$ is bounded by $\|f_j\|_{U^2}^{1/2}$.  More generally, we can hope to apply Cauchy--Schwarz repeatedly and systematically, eventually arriving at a system with Cauchy--Schwarz complexity $1$.

On the other hand, we show that the quantity $C(\Phi)$, which quantifies the number of times the Cauchy--Schwarz inequality is used, must necessarily grow without bound as $\Phi$ varies.

\begin{theorem}
  \label{thm:negative}
  For any sufficiently large prime $p$, $p \equiv \pm 1 \pmod{8}$, there exist a system $\Phi$ of $6$ linear forms in $3$ variables with $s(\Phi) = 1$, and functions $f_1,\dots,f_6 \colon \ZZ/p\ZZ \to \CC$ with $\|f_i\|_\infty \le 1$ for each $i$, such that
  \[
    \left|\Lambda_\Phi(f_1,\dots,f_6)\right| \ge 10^{-12}
  \]
but
  \[
    \|f_j\|_{U^2} \le p^{-1/8}
  \]
for each $j$.
\end{theorem}
Unlike in Theorem \ref{thm:positive}, here the system $\Psi$ is allowed to change as $p$ grows, with no control on the size of its coefficients.
The condition $p \equiv \pm 1 \pmod{8}$ is an inessential one related to the precise construction used and could be removed without too much added difficulty.

\begin{remark}
This negative result perhaps sheds some light on where the obstructions to a very straightforward proof of Theorem \ref{thm:positive} lie.

The difficulty turns out not to be that the Cauchy--Schwarz inequality is insufficiently powerful, or too blunt to detect the algebraic nature of the boundary between systems with $s=1$ and $s=2$; in fact it handles such considerations surprisingly easily.

Instead, the issue is that the Cauchy--Schwarz steps used must necessarily be tailor-made to the system $\Phi$ being considered.  The task of describing a mapping from systems to Cauchy--Schwarz arguments could be likened to that of building a primitive computer using only the Cauchy--Schwarz inequality.  Setting up the technical machinery required to achieve this will occupy most of the paper. 
\end{remark}

\begin{remark}
  The value of $C(\Phi)$ given by the proof of Theorem \ref{thm:positive} is completely explicit but in many cases unreasonably large. No serious attempt has been made to optimize it, although minor changes would probably produce only minor improvements.

  For large $p$, the worst-case behavior given by the proof is something like $\exp(O(K^{O(1)}))$ where $K$ is the size of the largest (integer) coefficient appearing in $\Phi$.  Although typically one expects not to hit the worst case, nonetheless in practice for integer coefficients of size about $10$ values such as $C \approx 2^{400}$ are not unusual.  It seems likely such values are not best possible.

  When $p$ is fixed, we may state a bound in terms of $p$ rather than the size of the coefficients.  Here the method gives $C(p) = O(p^{O(1)})$.  It is possible one could modify the argument to improve this to $O(\log p)$, which would be best possible up to absolute constants.  However, significant additional technical challenges arise, and so we will not attempt this here.
\end{remark}

\begin{remark}
  The general case of Questions \ref{q:bounds}, \ref{q:elementary} and \ref{q:coefficients}, for $d > 3$ or $r > 6$, remains open.  It seems reasonable to speculate that Theorem \ref{thm:positive} (and Theorem \ref{thm:negative}) have analogues in this level of generality.  There is no immediately apparent obstruction to the overall approach of repeated application of the Cauchy--Schwarz inequality succeeding in general, but conversely it is not obvious how to generalize the specific strategies used when $d=3$ and $r=6$ to the general case.  Therefore, this is left to possible future work.
\end{remark}

\subsection{Outline of the paper}

In Section \ref{sec:63} we present some preliminaries concerning the case of $6$ forms in $3$ variables.  In particular, we will deal with some initial degenerate cases where, in a technical and slightly disingenuous sense, we will see that applying the Cauchy--Schwarz inequality causes $s_{\CS}$ to decrease.  We will need these cases in what follows, but this also serves as an introduction to the general approach behind the proof of Theorem \ref{thm:positive} without the notational complexities.

In Section \ref{sec:formalisms} we introduce formalisms to keep track of the effects of multiple applications of Cauchy--Schwarz in a systematic manner.  This has the effect of reducing the proof of Theorem \ref{thm:positive} in any given instance, to winning a Cauchy--Schwarz ``game'' which has a well-defined set of possible moves and which can readily be simulated on a computer.

Section \ref{sec:hard-bit} addresses the core problem of solving this game in general.  This comes down to finding sequences of moves which have the effect of implementing predictable arithmetic operations on the system $\Phi$, and using them to walk $\Phi$ to a degenerate configuration of the type considered in Section \ref{sec:63}.

Finally, Section \ref{sec:negative} gives the proof of the negative result, Theorem \ref{thm:negative}.

\subsection{Notation}

We use $O(1)$ to denote any quantity bounded above by an absolute constant, and $O(X)$ to mean $O(1) X$.  The notation $[m]$ for $m$ a positive integer denotes the set $\{1,2,\dots,m\}$.  For a real parameter $x$, $e(x)$ denotes $\exp(2 \pi i x)$.  The notation $[A=B]$ (for example) denotes the indicator function of the event $A=B$.  If $W$ is a finite-dimensional vector space over $\FF_p$, we write $W^\ast$ for its dual space and $\PP(W)$ for the corresponding projective space (i.e.~the space of $1$-dimensional subspaces of $W$). Also, $\PP^k = \PP^k(\FF_p)$ means the same as $\PP\big(\FF_p^{k+1}\big)$.  Given $w \in W \setminus \{0\}$, we write $[w]$ for the corresponding element of $\PP(W)$.  If $U \subseteq W$ is a subspace of $W$, we write $U^\perp \subseteq W^\ast$ for the perpendicular subpsace, i.e.\ the set of all $\phi \in W^\ast$ that vanish on $U$.

\subsection{Acknowledgements}

The author would like to thank Sean Eberhard, Ben Green, Rudi Mrazovi\'c and Julia Wolf for discussions on these topics at various times.

\section{Preliminaries concerning six forms in three variables}
\label{sec:63}

We start by giving a brief analysis of the different cases that can arise concerning a system of six forms $\Phi = (\phi_1,\dots,\phi_6)$ in three variables, and the associated Cauchy--Schwarz complexity and true complexity.  Throughout this section we write $V = \FF_p^3$, so $\phi_i$ (modulo $p$) can be thought of as linear functionals $V \to \FF_p$, always assumed to be non-zero.

It is clear that nothing substantial changes when we replace $\phi_i$ by a non-zero scalar multiple $\lambda \phi_i$.  Indeed, the quantities $f_i(\phi_i(v))$ and $f_i(\lambda \phi_i(v))$ are essentially the same, up to replacing $f_i$ with a dilate of itself, and so this has no effect on the conclusion of Theorem \ref{thm:positive}; and by inspection our definitions of true complexity and Cauchy--Schwarz complexity are also unchanged.

Therefore it makes sense to think of the forms $\phi_i \colon V \to \FF_p$ as points $[\phi_i]$ in the projective plane $\PP(V^\ast) \cong \PP^2 (\FF_p)$, quotienting out by the action of scalar multiplication.  This allows us to phrase the different cases geometrically.

We have said that $\Phi$ has true complexity $s=1$ if the symmetric bilinear forms $\phi_1^2,\dots,\phi_6^2 \in (V^\ast)^{\otimes 2}$ are linearly independent.  Note that this space of symmetric bilinear forms on $V$ has dimension $(3\cdot4) / 2 = 6$, and there are six forms, so we expect this to be true generically.  Indeed, a dependence relation on $\phi_1^2,\dots,\phi_6^2$ exists if and only if there is a non-zero linear functional $\mu \colon \{ \sigma \in (V^\ast)^{\otimes 2} \colon \sigma = \sigma^T\} \to \FF_p$ which evaluates to $0$ on each $\phi_i^2$; and this in turn is the same thing as a non-zero quadratic form $V^\ast \to \FF_p$ which vanishes at each $\phi_i$; i.e., a conic in the projective plane $\PP(V^\ast)$ containing $[\phi_i]$ for each $i$.\footnote{Note that this argument is still valid as stated when $p=2$.}

In other words, we have shown the following.
\begin{slogan}
  \label{slogan:six-forms}
  Six forms $\phi_1,\dots,\phi_6$ on $\FF_p^3$ have true complexity at least $2$, if and only if $[\phi_1],\dots,[\phi_6]$ all lie on a \uppar{possibly degenerate} conic in $\PP(V^\ast)$.
\end{slogan}
By a degenerate conic, we mean the union of two lines.  In particular, if $[\phi_1],[\phi_2],[\phi_3]$ are collinear and so are $[\phi_4],[\phi_5],[\phi_6]$, then this system has true complexity at least $2$.

It is possible for the true complexity to be greater than $2$: for instance, if five of the points lie on a line in $\PP^2$, in which case the true complexity is $3$; or if $[\phi_i] = [\phi_j]$ for some $i \ne j$, in which case the system has infinite complexity.

However, all such cases may be fully analyzed in terms of Cauchy--Schwarz complexity, which gives a bound $|\Lambda_\Phi(f_1,\dots,f_6)| \le \|f_j\|_{U^{s_j+1}}$ for each $j$ where the values $s_j$ are best possible, even when they vary with $j$.  The details are an uninteresting check that will not be relevant to the argument, so are omitted.

%However, all such cases may be fully analyzed in terms of Cauchy--Schwarz complexity, which gives a bound $|\Lambda_\Phi(f_1,\dots,f_6) \le \|f_i\|_{U^{s_i+1}}$ for each $i$ where the values $s_i$ are best possible, even when they are different.  We will check only one instance of this, which we will need; the details of the remaining cases are an uninteresting check that will not be directly relevant, so are omitted.
%
%\begin{lemma}
%  \label{lem:two-the-same}
%  Let $\Phi = (\phi_1,\dots,\phi_6)$ on $\FF_p^3$ be a system of linear forms, and suppose $[\phi_5] = [\phi_6]$ but no three of $[\phi_1],\dots,[\phi_5]$ are collinear.  Then for functions $f_1,\dots,f_6 \colon \FF_p^n \to \CC$, with $\|f_i\|_\infty \le 1$ for each $i$, and for any $j=1,2,3,4$ we have a bound
%  \[
%    \left| \Lambda_\Phi(f_1,\dots,f_6) \right| \le \|f_j\|_{U^2} \, .
%  \]
%\end{lemma}
%\begin{proof}
%  For $j=1$ (the other cases being symmetric) partition $\phi_2,\dots,\phi_6$ as $\{\phi_2,\phi_3\}$ and $\{\phi_4,\phi_5,\phi_6\}$ (say).  It is clear that $\spn(\phi_4,\phi_5,\phi_6) = \spn(\phi_4,\phi_5)$, and therefore the span of neither set contains $\phi_1$ as all triples of $\phi_1,\dots,\phi_5$ are linearly independent. The bound then follows from \cite{???}.
%\end{proof}

We therefore restrict our attention to the cases with $s=1$.  In particular, we can henceforth make the following assumptions:
\begin{enumerate}[label=(\roman*)]
  \item the points $[\phi_1],\dots,[\phi_6]$ are all distinct;
  \item no four are collinear; and
  \item if some three of the points are collinear, the remaining three are \emph{not} collinear.
\end{enumerate}

It is clear that if the six forms are in general position, meaning no three are collinear, then $s_{\CS} = 2$.  Indeed, any way we partition all but one of the forms into two classes, one of the classes will contain three forms and so their span will be all of $V^\ast$; hence $s_{\CS} > 1$.  Conversely any $2-2-1$ split achieves $s_{\CS} \le 2$.

Our remaining task in this section is to consider the case where (i)--(iii) hold, but nonetheless $[\phi_1],\dots,[\phi_6]$ are not in general position.  This is a setting in which Cauchy--Schwarz complexity has some purchase, but nonetheless there is a subtlety meaning, technically speaking, that typically $s_{\CS} \ne 1$.

\begin{proposition}
  \label{prop:63-trivial-cs}
  Suppose throughout that a system of forms $\Phi = (\phi_1,\dots,\phi_6)$ on $\FF_p^3$ is given, with no four of $[\phi_1],\dots,[\phi_6]$ collinear and no two the same.
  \begin{enumerate}[label=(\roman*)]
    \item Suppose that $[\phi_1],[\phi_2],[\phi_3]$ are collinear but $[\phi_4],[\phi_5],[\phi_6]$ are not collinear.  Then for functions $f_1,\dots,f_6 \colon \FF_p^n \to \CC$, with $\|f_i\|_\infty \le 1$ for each $i$, and for any $j=4,5,6$ we have a bound
  \[
    \left| \Lambda_\Phi(f_1,\dots,f_6) \right| \le \|f_j\|_{U^2}
  \]
coming from Proposition \ref{prop:csc}.
    \item Under the same conditions as (i), the system has true complexity $s = 1$.  In particular, by results of Gowers and Wolf, $|\Lambda_\Phi(f_1,\dots,f_6)|$ is bounded in terms of $\|f_j\|_{U^2}$ for $j=1,2,3$ \uppar{at least for $n=1$ or $p$ fixed}.
    \item Now suppose further that $[\phi_1],[\phi_2],[\phi_3]$ is the only collinear triple. Then for $j=1,2,3$ there is no way to partition $\{1,\dots,6\} \setminus \{j\}$ into two pieces such that $\phi_j$ is in the span of neither piece, and hence $s_{\CS} = 2$ for this system.
  \end{enumerate}
\end{proposition}
\begin{proof}
  For (i), say when $j=6$, we can partition $\{1,\dots,5\}$ into $\{1,2,3\}$ and $\{4,5\}$.  By our assumptions, it is clear that $\phi_6$ is in neither $\spn(\phi_1,\phi_2,\phi_3) = \spn(\phi_1,\phi_2) \subseteq V^\ast$ nor $\spn(\phi_4,\phi_5) \subseteq V^\ast$, and so the bound indeed follows from Proposition \ref{prop:csc}.  The other choices of $j$ are analogous.

  For (ii), we note that a conic containing three distinct collinear points must be degenerate, but $[\phi_1],\dots,[\phi_6]$ are not contained in the union of any two lines.  Hence the points do not lie on a conic, and so $s=1$.

  For (iii), when say $j=1$, given any partition of $\{2,\dots,6\}$ into two pieces, one of the pieces contains three of the forms.  Since that triple is not $\phi_1,\phi_2,\phi_3$, they are not collinear and so their span is all of $V^\ast$.
\end{proof}

So, this is a case where $s$ and $s_{\CS}$ differ, albeit for what feels like a bad reason.  Indeed, it is not too challenging to recover a good bound on $\Lambda_\Phi(f_1,\dots,f_6)$ in terms of $\|f_1\|_{U^2}$ in this setting, for instance by decomposing $f_6$ into two parts corresponding to its large and small Fourier coefficients, bounding away the uniform contribution and treating what is left as essentially a system of five forms.

Instead, we will now recover such a bound purely by using the Cauchy--Schwarz inequality, and thereby provide the first (admittedly unimpressive) instantiation of Slogan \ref{slogan:cs}.

\begin{proposition}
  \label{prop:skew-collinear}
  Let $\Phi = (\phi_1,\dots,\phi_6)$ be a system of linear forms on $\FF_p^3$, such that
    no four of $[\phi_1],\dots,[\phi_6]$ are collinear and no two are the same;
    $[\phi_1],[\phi_2],[\phi_3]$ are collinear;
    and $[\phi_4],[\phi_5],[\phi_6]$ are not collinear.

  Then for functions $f_1,\dots,f_6 \colon \FF_p^n \to \CC$, with $\|f_i\|_\infty \le 1$ for each $i$, we have a bound
  \[
    \left| \Lambda_\Phi(f_1,\dots,f_6) \right| \le \|f_1\|_{U^2}^{1/2} \, .
  \]
\end{proposition}
\begin{proof}
  By applying a suitable change of basis to $V = \FF_p^3$, we may assume without loss of generality that $\phi_6(x,y,z) = x$;  this is not essential but eases the notation.  So,
  \[
    \Lambda_\Phi(f_1,\dots,f_6) = \EEE_{x \in \FF_p} f_6(x) \EEE_{y,z \in \FF_p} f_1(\phi_1(x,y,z)) \dots f_5(\phi_5(x,y,z)) .
  \]
  We can apply the Cauchy--Schwarz inequality to obtain
  \[
    \left|\Lambda_\Phi(f_1,\dots,f_6)\right| \le \left(\EEE_x |f_6(x)|^2 \right)^{1/2} \left( \EEE_x \left| \EEE_{y,z} f_1(\phi_1(x,y,z)) \dots f_5(\phi_5(x,y,z)) \right|^2 \right)^{1/2} .
  \]
  Now, the term on the right expands to
  \[
    \EEE_{x,y,y',z,z'} f_1(\phi_1(x,y,z)) \dots f_5(\phi_5(x,y,z)) \overline{f_1(\phi_1(x,y',z')) \dots f_5(\phi_5(x,y',z'))}
  \]
  and we can think of this as $\Lambda_{\Phi'}(f_1,\dots,f_5,\overline{f_1},\dots,\overline{f_5})$ where $\Phi'$ is the system of $10$ forms in the five variables $x,y,y',z,z'$, given by $\phi_{i_0}(x,y,y',z,z') = \phi_i(x,y,z)$ and $\phi_{i_1}(x,y,y',z,z') = \phi_i(x,y',z')$ for $1 \le i \le 5$, each thought of as a linear functional $\FF_p^5 \to \FF_p$.

  We claim that, under our hypotheses, it is possible to partition the nine forms $\phi_{2_0}$, $\phi_{3_0}$, $\dots$, $\phi_{5_0}$, $\phi_{1_1}$, $\dots$, $\phi_{5_1}$ into two classes such that $\phi_{1_0}$ is not in the span of either class.  Specifically, we will take
  \begin{align*}
    S &= \{ 2_0, 4_0, 1_1, 2_1, 3_1 \} \\
    T &= \{ 3_0, 5_0, 4_1, 5_1 \}
  \end{align*}
  to be the sets of indices in each class.  If this claim holds, then by the standard Cauchy--Schwarz complexity bound (Proposition \ref{prop:csc}) again we have
  \[
    \left| \Lambda_{\Phi'}(f_1,\dots,f_5,\overline{f_1},\dots,\overline{f_5}) \right| \le \|f_1\|_{U^2}
  \]
  and the result follows.

  We now verify the claim.  We first note that, since $[\phi_1],[\phi_2],[\phi_3]$ are collinear by hypothesis,
  \[
    \spn(\phi_{2_0}, \phi_{4_0}, \phi_{1_1}, \phi_{2_1}, \phi_{3_1}) = \spn(\phi_{2_0}, \phi_{4_0}, \phi_{1_1}, \phi_{2_1})
  \]
as $\phi_{3_1} \in \spn(\phi_{1_1}, \phi_{2_1})$.
  This makes the claim plausible for dimension reasons: it is reasonable to expect the span of four linear forms on $\FF_p^5$ not to contain a fifth, unless something untoward happens.  However, something untoward could genuinely happen if too many of the original forms are collinear, and more generally we need to show that all bad cases are ruled out by our hypotheses.  This is the technical part of the calculation, and may be skipped on first (or subsequent) reading.

  Recall $\phi_6(x,y,z) = x$ and write $\phi_i(x,y,z) = a_i x + b_i y + c_i z$ for $1 \le i \le 5$.
  To show $\phi_{1_0}$ is not in the span of $\phi_{2_0},\phi_{4_0},\phi_{1_1},\phi_{2_1}$, it would suffice to show that $\phi_{1_0}$ together with the other four form a basis for $\left(\FF_p^5\right)^\ast$; equivalently, that the matrix
  \[
    M_S = \begin{pmatrix}
      a_1 & b_1 & 0 & c_1 & 0 \\
      a_2 & b_2 & 0 & c_2 & 0 \\
      a_4 & b_4 & 0 & c_4 & 0 \\
      a_1 & 0 & b_1 & 0 & c_1 \\
      a_2 & 0 & b_2 & 0 & c_2
    \end{pmatrix}
  \]
  (whose columns correspond to $x,y,y',z,z'$ respectively) is non-singular.  However, it is not hard to see that
  \begin{align*}
    \det M_S &= \pm \det \begin{pmatrix} 
      a_1 & b_1 & c_1 \\
      a_2 & b_2 & c_2 \\
      a_4 & b_4 & c_4
    \end{pmatrix}\ \det \begin{pmatrix}
      b_1 & c_1 \\
      b_2 & c_2
    \end{pmatrix} \\
    &= 
    \pm \det \begin{pmatrix} 
      a_1 & b_1 & c_1 \\
      a_2 & b_2 & c_2 \\
      a_4 & b_4 & c_4
    \end{pmatrix}\ \det \begin{pmatrix}
      1 & 0 & 0 \\
      a_1 & b_1 & c_1 \\
      a_2 & b_2 & c_2
    \end{pmatrix}\, .
  \end{align*}
  The determinants on the right hand side are zero precisely when, respectively, $[\phi_1],[\phi_2],[\phi_4]$ or $[\phi_6], [\phi_1],[\phi_2]$ are collinear.  Under our assumptions, neither can be true (as then four points would lie on a line) and so $M_S$ is non-singular.

  The argument for $T$ is very similar.  We define
  \[
    M_T = \begin{pmatrix}
      a_1 & b_1 & 0 & c_1 & 0 \\
      a_3 & b_3 & 0 & c_3 & 0 \\
      a_5 & b_5 & 0 & c_5 & 0 \\
      a_4 & 0 & b_4 & 0 & c_4 \\
      a_5 & 0 & b_5 & 0 & c_5
    \end{pmatrix}
  \]
  which is non-singular if and only if $\phi_{1_0},\phi_{3_0},\phi_{5_0},\phi_{4_1},\phi_{5_1}$ form a basis.  Then
  \[
    \det M_T = 
    \pm \det \begin{pmatrix} 
      a_1 & b_1 & c_1 \\
      a_3 & b_3 & c_3 \\
      a_5 & b_5 & c_5
    \end{pmatrix}\ \det \begin{pmatrix}
      1 & 0 & 0 \\
      a_4 & b_4 & c_4 \\
      a_5 & b_5 & c_5
    \end{pmatrix}
  \]
  and again this is zero if and only if either $[\phi_1],[\phi_3],[\phi_5]$ or $[\phi_6],[\phi_4],[\phi_5]$ are collinear.  Again, both of these are explicitly ruled out by our hypotheses, and this proves the claim.
\end{proof}

\begin{remark}
  One way to think of this proof on a high level is as a combinatorial analogue of the method we sketched above: namely, first observing that $\Lambda_{\Phi}$ is controlled by $\|f_6\|_{U^2}$, then noting this allows us to essentially eliminate $f_6$ by replacing it with the sum of its large Fourier coefficients, and finally applying Cauchy--Schwarz on the remaining five forms.

  What we do here is first make two copies of the original system, joined by $\phi_6$; on the right, we decompose the remaining forms as if we were attempting to prove a Cauchy--Schwarz complexity bound in $\|f_6\|_{U^2}$, as in Proposition \ref{prop:63-trivial-cs}; and on the left we decompose as if we were  tring to prove a Cauchy--Schwarz complexity bound in $\|f_1\|_{U^2}$ and the form $\phi_6$ didn't exist.

  So, the initial Cauchy--Schwarz allows us to somehow substitute the information gained from the former argument into the latter.
\end{remark}

\begin{remark}
  As we have said, this is an application of Slogan \ref{slogan:cs}, but not a very convincing one.  Before embarking on the programme in full generality, we briefly sketch an example of six forms in general position, having $s=1$ (but necessarily $s_{\CS} = 2$), where we nonetheless get a bound using only Cauchy--Schwarz.

  Consider the forms
  \begin{align*}
    \phi_1(x,y,z) &= x &
    \phi_4(x,y,z) &= x+y+z \\
    \phi_2(x,y,z) &= y &
    \phi_5(x,y,z) &= 2x+3y+5z \\
    \phi_3(x,y,z) &= z &
    \phi_6(x,y,z) &= ax+by+cz
  \end{align*}
  where $a,b,c \in \FF_p$ are arbitrary subject to the condition that the forms be in general position.  For concreteness one could substitute $a=7$, $b=11$, $c=13$.

  We can apply Cauchy--Schwarz to $\phi_1,\dots,\phi_6$ twice as follows:
  \begin{align*}
    &\left| \EEE_{x,y,z} f_1(x) f_2(y) f_3(z) f_4(x+y+z) f_5(2x+3y+5z) f_6(ax+by+cz) \right|  \\
    \ &= \left| \EEE_x f_1(x) \EEE_{y,z} f_2(y) f_3(z) f_4(x+y+z) f_5(2x+3y+5z) f_6(ax+by+cz) \right| \\ 
    \ &\le \left( \EEE_x |f_1(x)|^2 \right)^{1/2} \bigg( \EEE_{x,y_0,y_1,z_0,z_1} \begin{aligned}[t] & f_2(y_0) f_3(z_0) f_4(x+y_0+z_0) f_5(2x+3y_0+5z_0) f_6(ax_0+by_0+cz_0) \\ & \overline{f_2(y_1) f_3(z_1) f_4(x+y_1+z_1) f_5(2x+3y_1+5z_1) f_6(ax+by_1+cz_1)} \bigg)^{1/2} \end{aligned}
  \end{align*}
  and then again
  \begin{align*}
    &\EEE_{y_0,y_1} f_2(y_0) \overline{f_2(y_1)} \EEE_{x,z_0,z_1} \begin{aligned}[t] & f_3(z_0) f_4(x+y_0+z_0) f_5(2x+3y_0+5z_0) f_6(ax+by_0+cz_0) \\ & \overline{f_3(z_1) f_4(x+y_1+z_1) f_5(2x+3y_1+5z_1) f_6(ax+by_1+cz_1)} \end{aligned} \\
      \ &\le \left( \EEE_{y_0,y_1} |f_2(y_0)|^2 |f_2(y_1)|^2 \right)^{1/2} \bigg( \EEE_{\substack{y_0,y_1,\\x_0,x_1,\\z_{00},z_{01},\\z_{10},z_{11}}} \begin{aligned}[t] & f_3(z_{00}) f_4(x_0+y_0+z_{00}) f_5(2x_0+3y_0+5z_{00}) f_6(ax_0+by_0+cz_{00}) \\ & \overline{f_3(z_{10}) f_4(x_0+y_1+z_{10}) f_5(2x_0+3y_1+5z_{10}) f_6(ax_0+by_1+cz_{10})} \\ & \overline{f_3(z_{01}) f_4(x_1+y_0+z_{01}) f_5(2x_1+3y_0+5z_{01}) f_6(ax_1+by_0+cz_{01})} \\ & f_3(z_{11}) f_4(x_1+y_1+z_{11}) f_5(2x_1+3y_1+5z_{11}) f_6(ax_1+by_1+cz_{11}) \bigg)^{1/2} . \end{aligned}
  \end{align*}
  We now claim that this last system of $16$ linear forms in $8$ variables has Cauchy--Schwarz complexity $1$ with respect to $f_3(z_{00})$, if and only if the original system has true complexity $1$.  That is, we can partition the remaining forms into two classes:
  \begin{align*}
    S &= \big \{ x_0+y_0+z_{00},\, 2x_0+3y_0+5z_{00},\, z_{10},\, x_0+y_1+z_{10},\, z_{01},\, x_1+y_0+z_{01},\, z_{11},\, 2x_1+3y_1+5z_{11} \big\} \\
    T &= \big \{ \begin{aligned}[t]  ax_0+&by_0+cz_{00},\, 2x_0+3y_1+5z_{10},\, ax_0+by_1+cz_{10},\, 2x_1+3y_0+5z_{01}, \\ & ax_1+by_0+cz_{01},\, x_1+y_1+z_1,\, ax_1+by_1+cz_{11} \big\} \end{aligned} 
  \end{align*}
  such that $z_{00}$ lies in the span of one of the classes, if and only if $[\phi_1],\dots,[\phi_6]$ lie on a conic.  Verifying this claim is left as an exercise for the interested reader.  We stress that for particular choices of $a,b,c$ this is an elementary finite computation.

  The fact that these kinds of linear algebra conditions can detect whether the points lie on a conic should perhaps not be surprising in light of Pascal's hexagon theorem.
\end{remark}

\section{Formalisms for iterated Cauchy--Schwarz}
\label{sec:formalisms}

The purpose of this section is to introduce some formalisms necessary to keep track of what happens when we apply the Cauchy--Schwarz inequality repeatedly.  The notational overhead here is high, but preferable to handling yet larger explicit calculations in the style of the previous section.

\subsection{Linear data}

Although the central objects of study are systems of linear forms, it will be convenient to use a natural generalization of this notion, which handles the objects that arise in intermediate stages of the calculation.  We introduce the relevant definitions now.

\begin{definition}
  Let a prime $p$ be fixed.  By a \emph{linear datum}, we mean a tuple $\Psi = \left(V, (W_i)_{i \in I}, (\psi_i)_{i\in I}\right)$, where $I$ is some finite index set, and
  \begin{itemize}
    \item $V$ and $W_i$ for $i \in I$ are finite-dimensional vector spaces over $\FF_p$;
    \item $\psi_i \colon V \to W_i$ are surjective linear maps.
  \end{itemize}

  Given a positive integer $n$, we abuse notation to write $\psi_i \colon V^n \to W_i^n$ for the map that applies $\psi_i$ to each coordinate.  Now, for a collection of functions $(f_i)_{i \in I}$ with $f_i \colon W_i^n \to \CC$, we define
  \[
    \Lambda_{\Psi}((f_i)_{i \in I}) = \EEE_{v \in V^n} \prod_{i \in I} f_i(\psi_i(v)) \, .
  \]
\end{definition}
It is clear that in the special case that $\dim W_i = 1$ for each $i$, this is essentially the same information as a system of linear forms on $V \cong \FF_p^d$ for some $d$.  The reader should always imagine $\dim V$ as being small, even when we are working over $G = \FF_p^n$ for some large $n$: the $n$ is taken care of in the definition of $\Lambda_\Psi$, not of $\Psi$.

Attempting to analyse linear data in general exposes hard problems; see \cite{austin1,austin2}. Since the linear data we will consider ultimately come from systems of linear forms, these subtleties will not arise here.

\begin{remark}
  Typically we are not too concerned by replacing $W_i$ by isomorphic vector spaces, or by the exact form of the linear map $\psi_i$: for instance, as we have said the difference between $f(\psi_i(v))$ and $f(2\psi_i(v))$ is usually immaterial.

  As such, the only really important information is the collection of subspaces $\ker \psi_i$ of $V$, as we can always recover $W_i$ up to isomorphism as $V / \ker \psi_i$.  One can interpret $\ker \psi_i$ as the subspace of $V$ that the function $f_i$ \emph{cannot} depend on.

  Alternatively, we could think about the perpendicular subspaces $(\ker \psi_i)^\perp \subseteq V^\ast$, corresponding to the span of all linear forms derived from $\psi_i$.  This is consistent with the geometric picture from Section \ref{sec:63}: such subspaces correspond to points, lines, planes etc.~in $\PP(V^\ast)$.

  For technical reasons it is useful to keep track of the linear maps $\psi_i$ explicitly; but the reader will rarely lose anything, and possibly gain something, by thinking of a linear datum as simply a collection of subspaces of $V$ or $V^\ast$.
\end{remark}

We need some notion of when one linear datum bounds another; for instance, but not exclusively, because one is obtained by applying the Cauchy--Schwarz inequality to the other.

\begin{definition}
  \label{def:dominate}
  Suppose we have two linear data $\Psi = \left(V, (W_i)_{i \in I}, (\psi_i)_{i \in I}\right)$ and $\Psi' = \left(V', (W'_i)_{i\in I'}, (\psi'_i)_{i \in I'}\right)$.  Suppose further that for some pair $j \in I$, $j' \in I'$ the subspaces $W'_{j'} = W_{j}$ are identified.  Finally, let $c > 0$ be a positive real number.

  We say \emph{$\Psi'$ dominates $\Psi$ respecting $(j,j')$ with exponent $c$} if the following holds: for $n \ge 1$ and any collection of functions $(f_i)_{i \in I}$, $f_i \colon W_i^n \to \CC$ with $\|f_i\|_\infty \le 1$, there exist functions $(g_i)_{i \in I'}$, $g_i \colon {W'_i}^n \to \CC$, $\|g_i\|_\infty \le 1$, such that $g_{j'} = f_{j}$, and
  \[
    \left| \Lambda_\Psi((f_i)_{i \in I}) \right| \le \left| \Lambda_{\Psi'}((g_i)_{i \in I'}) \right|^{c} .
  \]
\end{definition}
It is clear that domination is transitive: if $\Psi'$ dominates $\Psi$ respecting $(j,j')$ with exponent $c$, and $\Psi''$ dominates $\Psi'$ respecting $(j',j'')$ with exponent $c'$, then $\Psi''$ dominates $\Psi$ respecting $(j,j'')$ with exponent $cc'$.

Some straightforward examples of domination include (i) replacing $\Psi$ by an isomorphic system (i.e.~reparameterizing); (ii)   augmenting $\Psi$ by introducing further averaging, or by replacing $\ker \psi_i$ by a strictly larger subspace for some $i$; or (iii) taking a supremum over some part of the average.  All of these are subsumed in the following general proposition.

\begin{proposition}
  \label{prop:trivial-dominate}
  Suppose $\Psi = \left(V, (W_i)_{i \in I}, (\psi_i)_{i \in I}\right)$ and $\Psi' = \left(V', (W'_i)_{i\in I}, (\psi'_i)_{i \in I}\right)$ are two linear data on the same index set $I$, and that we are given linear maps $\theta \colon V' \to V$ and $\sigma_i \colon W_i' \to W_i$ such that $\psi_i \circ \theta = \sigma_i \circ \psi'_i$ for each $i \in I$ \uppar{i.e.~a morphism of linear data}.  If $j \in I$ is some index such that $W_j = W'_j$ and $\sigma_j$ is the identity, then $\Psi'$ dominates $\Psi$ respecting $(j,j)$, with exponent $1$.
\end{proposition}
\begin{proof}
  Let $v_1,\dots,v_M \in V^n$ be a complete set of coset representatives of $\theta(V')^n$ in $V^n$.  By our hypotheses, we have $\psi_j(\theta(V')) = \psi'_j(V') = W_j$ and so $\ker \psi_j + \theta(V') = V$; hence we may insist that $v_1,\dots,v_M$ all lie in $(\ker \psi_j)^n$.
  
  For any collection of functions $f_i \colon W_i^n \to \CC$, $i \in I$, we have
  \begin{align*}
    \left| \Lambda_\Psi((f_i)_{i\in I}) \right|
      &= \left| \EEE_{v \in {V}^n} \prod_{i \in I} f_i(\psi_i(v)) \right| \\
      &= \left| \EEE_{\ell \in [M]} \EEE_{v' \in {V'}^n} \prod_{i \in I} f_i\left(\psi_i(v_\ell + \theta(v'))\right) \right| \\
      &\le \max_{\ell \in [M]} \left| \EEE_{v' \in {V'}^n} \prod_{i \in I} f_i\left(\psi_i(v_\ell) + \sigma_i(\psi'_i(v'))\right) \right| \, .
  \end{align*}
  Now fix $\ell$ to be any maximal choice, and define $g_i \colon {W'_i}^n \to \CC$ by
  \[
    g_i(w) = f_i(\psi_i(v_\ell) + \sigma_i(w)) \, .
  \]
  We deduce that $\left| \Lambda_\Psi((f_i)_{i \in I}) \right| \le \left| \Lambda_{\Psi'}((g_i)_{i \in I}) \right|$.  Moreover, it follows from our assumptions that $g_j = f_j$, and so the conditions of Definition \ref{def:dominate} are satisfied.
\end{proof}

\begin{definition}
  \label{def:augment}
  If $(\Psi, \Psi')$ obey the hypotheses of Proposition \ref{prop:trivial-dominate}, we say $\Psi'$ dominates $\Psi$ \emph{trivially} at index $j$.  Replacing $\Psi$ by $\Psi'$ is termed a $\TRIVIAL$ operation.
\end{definition}

We now consider how to describe an application of the Cauchy--Schwarz inequality in this language.

\begin{proposition}
  \label{prop:cs}
  Suppose $\Psi = \left(V, (W_i)_{i \in I}, (\psi_i)_{i \in I}\right)$ is a linear datum, and some $j \in I$ is given.  Let $\Psi' = \left(V', (W'_i)_{i\in I'}, (\psi'_i)_{i \in I'}\right)$ be the linear datum defined as follows:
  \begin{itemize}
    \item $V'$ is the fiber product of $V$ with itself over $W_j$, i.e.:
      \[
        V' = \left\{ (v_0, v_1) \in V \oplus V \colon \psi_j(v_0) = \psi_j(v_1) \right\} ;
      \]
    \item $I'$ is the disjoint union of two copies of $I \setminus \{j\}$, denoted
      \[
        I' = \{ i_0 \colon i \in I, i \ne j \} \CMamalg \{ i_1 \colon i \in I, i \ne j \} ;
      \]
    \item for each $i \in I \setminus \{j\}$, $W'_{i_0} = W'_{i_1} = W_i$; and
    \item for each $i \in I \setminus \{j\}$ and $v' = (v_0,v_1) \in V'$,
      \begin{align*}
        \psi'_{i_0}(v_0,v_1) &= \psi_i(v_0) \\
        \psi'_{i_1}(v_0,v_1) &= \psi_i(v_1) .
      \end{align*}
  \end{itemize}
  Then for any $i \in I \setminus \{j\}$, $\Psi'$ dominates $\Psi$ respecting $(i, i_0)$ and with exponent $1/2$.
\end{proposition}
We note that $\psi'_{i_0}$, $\psi'_{i_1}$ are surjective, e.g.~by observing that $\{(v,v) \colon v \in V\}$ is a subspace of $V'$.
\begin{proof}
  As promised, this is just the statement of the Cauchy--Schwarz inequality as it applies in this context.  
  Given $f_i \colon W_i^n \to \CC$, we have
  \begin{align*}
    \left| \Lambda_{\Psi}((f_i)_{i \in I}) \right| &= \left| \EEE_{w \in W_j} f_j(w) \EEE_{ v \in \psi_j^{-1}(w) } \prod_{i \ne j} f_i(\psi_i(v)) \right| \\
      &\le \left( \EEE_{w \in W_j} |f_j(w)|^2 \right)^{1/2} \left( \EEE_{w \in W_j} \left| \EEE_{ v \in \psi_j^{-1}(w) } \prod_{i \ne j} f_i(\psi_i(v)) \right|^2 \right)^{1/2}
  \end{align*}
  by Cauchy--Schwarz, and
  \begin{align*}
    &\EEE_{w \in W_j} \left| \EEE_{v \in \psi_j^{-1}(w) } \prod_{i \ne j} f_i(\psi_i(v)) \right|^2 \\
    =& \EEE_{w \in W_j} \left(\EEE_{v_0 \in \psi_j^{-1}(w) } \prod_{i \ne j} f_i(\psi_i(v_0)) \right) \left( \EEE_{v_1 \in \psi_j^{-1}(w) } \prod_{i \ne j} \overline{f_i(\psi_i(v_1))} \right) \\
        =& \Lambda_{\Psi'}\left((f_i)_{i \in I \setminus \{j\}}, (\overline{f_i})_{i \in I \setminus \{j\}} \right) .
  \end{align*}
  Defining $(g_{i})_{i \in I'}$ in the obvious way, and provided $\|f_i\|_\infty \le 1$ for each $i \in I$, we get the desired inequality.
\end{proof}

\begin{definition}
  We denote the system $\Psi'$ defined in Proposition \ref{prop:cs} by $\CS_j(\Psi)$.
\end{definition}

Often we need to apply Cauchy--Schwarz not just to one function, but to several at a time.  The preferred way of formalizing this for our purposes is in two steps.  First, we \emph{merge} all the functions being considered for Cauchy--Schwarz into a single function.  That is, we forget that they are separate functions, and consider their product as just one function of all the variables they collectively depend on.  For instance, we might merge $f_1(x)$ and $f_2(x+y)$ into $\cF(x,y)$.  Next, we apply the Cauchy--Schwarz inequality in the form of Proposition \ref{prop:cs} to the new function $\cF$.

In fact we will want to apply this merging operation in other contexts as well, because doing so is one way to eliminate redundant information.  Having this ability is one of the main motivations for working in this more general language of linear data. 

Again, we encode this operation with a proposition.

\begin{proposition}
  \label{prop:merge}
  Let $\Psi = \left(V, (W_i)_{i \in I}, (\psi_i)_{i \in I}\right)$ be a linear datum, let $J$ be a finite set, and let $\tau \colon I \to J$ be a surjective function.  Define a new linear datum $\Psi' = \left(V, (W'_j)_{j \in J}, (\psi'_j)_{j \in J}\right)$ on the same underlying space $V$, as follows:
  \begin{itemize}
    \item for each $j \in J$, define
      \begin{align*}
        \psi'_j \colon V &\to \bigoplus_{i \in \tau^{-1}(j)} W_i \\
        v &\mapsto (\psi_i(v))_{i \in \tau^{-1}(j)} ;
      \end{align*}
    \item define $W'_j = \im \psi'_j$; and
    \item by abuse of notation consider $\psi'_j$ as a map $V \to W'_j$.
  \end{itemize}
  Then for any $i \in I$ and $j \in J$ with $\tau^{-1}(j) = \{i\}$, $\Psi'$ dominates $\Psi$ respecting $(i,j)$ and with exponent $1$.
\end{proposition}
\begin{proof}
  Given $(f_i)_{i \in I}$, for each $j \in J$ let
  \begin{align*}
    g_j \colon \bigoplus_{i \in \tau^{-1}(j)} W_i^n &\to \CC \\
          (w_i) &\mapsto \prod_{i \in \tau^{-1}(j)} f_i(w_i)
  \end{align*}
  and then restrict this function to the subspace ${W'_j}^n$.  Then it is easy to see that
  \[
    \Lambda_{\Psi}\left((f_i)_{i \in I}\right) = \Lambda_{\Psi'}\left((g_j)_{j \in J} \right)
  \]
and so the necessary inequality is in fact an equality.  Moreover, if $\tau^{-1}(j) = \{i\}$ is a singleton then $W_i = W'_j$ and $g_j = f_i$, so the conditions of Definition \ref{def:dominate} are met.
\end{proof}

\begin{remark}
  The definitions of $\psi'_j$ and $W'_j$ are somewhat involved. A more natural characterization in terms of subspaces of $V$ is that
  \[
    \ker \psi'_j = \bigcap_{i \in \tau^{-1}(j)} \ker \psi_j
  \]
  i.e.~merging functions corresponds to intersecting the corresponding subspaces.  Dually, we have
  \[
    (\ker \psi'_j)^\perp = \sum_{i \in \tau^{-1}(j)} (\ker \psi_j)^{\perp}
  \]
  so merging takes spans of the relevant subspaces of $V^\ast$.  Again, either of these allows us to reconstruct $\psi'_j$ and $W'_j$ up to isomorphism.
\end{remark}

\begin{definition}
  We denote the linear datum $\Psi'$ defined as in Proposition \ref{prop:merge} by $\MERGE_\tau(\Psi)$.
\end{definition}

In a slight abuse of notation, we may omit any indices that are unchanged by $\tau$ from the description of $\tau$.  For instance, the operation that merges indices $4$ and $7$ and labels the new combined index $A$ might be denoted $\MERGE_{\{4,7\} \mapsto A}$.

We are now in a position to state a version of Theorem \ref{thm:positive} coded in this language. 

\begin{lemma}
  \label{lem:game}
  Let $p$ be a prime, and let $\phi_1,\dots,\phi_6 \colon \ZZ^3 \to \ZZ$ be six linear forms in three variables.  Let $V = \FF_p^3$, and by abuse of notation let $\phi_i \colon V \to \FF_p$ for $1 \le i \le 6$ denote the same forms reduced modulo $p$, assumed to be non-zero.  Write $W_i = \FF_p$ for $1 \le i \le 6$, set $I = [6]$, and hence define the linear datum $\Psi = \left(V, (W_i)_{i \in I}, (\phi_i)_{i \in I}\right)$.

  Suppose $[\phi_1],\dots,[\phi_6]$ do not lie on a conic in $\PP(V^\ast)$.  Then there is some sequence of operations $\TRIVIAL$, $\CS_j$ and $\MERGE_\tau$ which can be applied to $\Psi$ in turn to produce a final linear datum $\Psi'$, such that:
  \begin{itemize}
    \item $\Psi' = (V, (W_i)_{i \in I}, (\phi'_i)_{i \in I})$ where $I$, $V$, $W_i$ are unchanged and $\phi_i = \phi_i'$ for $1 \le i \le 4$; i.e.~$\Psi'$ again corresponds to $6$ linear forms in $3$ variables over $\FF_p$, where $\phi'_5$ and $\phi'_6$ only may have changed;
    \item by applying Propositions \ref{prop:trivial-dominate}, \ref{prop:cs} or \ref{prop:merge} as appropriate as we go, we can deduce that $\Psi'$ dominates $\Psi$ respecting $(1,1)$ and with exponent $2^{-m}$ where $m$ is the number of $\CS$ steps;
    \item $m$ is bounded by $O(K^{O(1)})$ where $K$ is the size of the largest coefficient of $\Phi$; or alternatively by $O(\log p)$; and
    \item the points $[\phi'_1],\dots,[\phi'_6]$ do not lie on a conic, but some three are collinear.
  \end{itemize}
\end{lemma}

This last condition means that one of Proposition \ref{prop:63-trivial-cs} or Proposition \ref{prop:skew-collinear} applies to the forms $\phi'_1,\dots,\phi'_6$, and so
\[
  |\Lambda_{\Psi'}(f_1,\dots,f_6)| \le \|f_1\|_{U^2}^{1/2}
\]
for any $f_i \colon W_i \to \CC$ with $\|f_i\|_\infty \le 1$.  Combining this with the domination statement allows us to deduce Theorem \ref{thm:positive} (at least for $j=1$; the other cases follow by relabelling the indices).

\begin{remark}
  The combinatorial operations $\CS_j$, $\MERGE_\tau$ describe the heart of any strategy, whereas $\TRIVIAL$ steps are really just book-keeping to aid with proofs.  One could in principle delay all $\TRIVIAL$ steps to the end of the argument, or perhaps remove them completely, without fundamentally changing the approach.
\end{remark}

\subsection{Graphs of vector spaces}

One remaining difficulty in reasoning about the effect of repeated invocations of $\CS_j$ is finding a good notation for discussing the iterated fiber products that arise in the definition of the ambient vector space $V'$.

At the expense of yet further notational overhead, we introduce one more tool to help with this.  This subsection has very little content beyond allowing us to draw certain diagrams and make sense of what they mean.

\begin{definition}
  Let $\cG = (X,E)$ be a (multi)-graph with vertex set $X$ and edge set $E$, and let $V$ be any vector space over $\FF_p$.  Suppose that to every edge $e = (x,y) \in E$ is associated a subspace $H_e$ of $V$.  Then the vector space $\cG(V,(H_e)_{e \in E})$ associated to this set-up is the subspace of $V^X$ given by
  \[
    \left\{ \rho \colon X \to V \colon \rho(x) - \rho(y) \in H_e \text{ for every } (x,y) = e \in E \right\} \, .
  \]
\end{definition}

In other words, we place a copy of $V$ at every vertex and impose a compatibility restriction for every edge.

We will always apply this when each subspace $H_e$ is one of $\ker \psi_i$ for $1 \le i \le 6$, where $\psi_i \colon V \to W_i$ is part of some original linear datum with underlying space $V$.  It then makes sense to label each edge $e$ with a number $i \in [6]$, in place of the subspace $H_e = \ker \psi_i$.

The useful feature of this set-up is that $\CS_j$ steps correspond to simple combinatorial operations on the graph $\cG$: we replace $X$ by two copies $X_0,X_1$, keeping all the edges in each half; and we add an edge between $X_0$ and $X_1$ for every linear form involved in that Cauchy--Schwarz step (which is applied to the merge of some linear forms).

This is best illustrated by example.  Suppose we start with a linear datum $\Psi = \left(V, (W_i)_{i \in [6]}, (\psi_i)_{i \in [6]}\right)$.  At this point, the graph $\cG$ consists of a single vertex and no edges.

\begin{center}
  \begin{tikzpicture}
    \node[circle,draw] (0) at (0,0) {};
  \end{tikzpicture}
\end{center}
If we now apply $\CS_6$, we get a linear datum whose underlying vector space corresponds to the following graph:
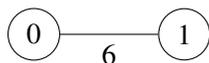
\begin{center}
  \begin{tikzpicture}
    \node[circle,draw] (0) at (0,0) {$0$};
    \node[circle,draw] (1) at (2,0) {$1$};
    \draw (0) -- (1) node[midway,below] {$6$};
  \end{tikzpicture}
\end{center}
Indeed, the definition of $\cG(V, (H_e)_{e \in E})$ in this case gives exactly the fiber product from Proposition \ref{prop:cs}.  Recall that the indices / forms in scope are now called $1_0,2_0,3_0,4_0,5_0$ and $1_1,2_1,3_1,4_1,5_1$ and are associated to the left vertex and the right vertex respectively.

Suppose we now apply $\MERGE_{\{4_0, 5_1\} \mapsto A}$ (recalling that by convention we assume the other indices are sent to themselves) and then apply $\CS_A$.  The new graph is:

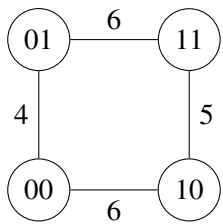
\begin{center}
  \begin{tikzpicture}
    \node[circle,draw] (00) at (0,0) {$00$};
    \node[circle,draw] (10) at (2,0) {$10$};
    \node[circle,draw] (01) at (0,2) {$01$};
    \node[circle,draw] (11) at (2,2) {$11$};
    \draw (00) -- (10) node[midway,below] {$6$};
    \draw (01) -- (11) node[midway,above] {$6$};
    \draw (00) -- (01) node[midway,left] {$4$};
    \draw (10) -- (11) node[midway,right] {$5$};
  \end{tikzpicture}
\end{center}

Finally, we might apply $\CS_{5_{01}}$ to obtain:

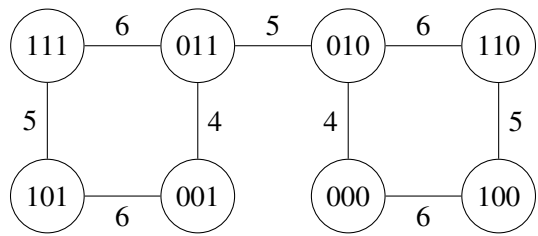
\begin{center}
  \begin{tikzpicture}
    \node[circle,draw] (000) at (0,0) {$000$};
    \node[circle,draw] (100) at (2,0) {$100$};
    \node[circle,draw] (010) at (0,2) {$010$};
    \node[circle,draw] (110) at (2,2) {$110$};
    \draw (000) -- (100) node[midway,below] {$6$};
    \draw (010) -- (110) node[midway,above] {$6$};
    \draw (000) -- (010) node[midway,left] {$4$};
    \draw (100) -- (110) node[midway,right] {$5$};

    \node[circle,draw] (001) at (-2,0) {$001$};
    \node[circle,draw] (101) at (-4,0) {$101$};
    \node[circle,draw] (011) at (-2,2) {$011$};
    \node[circle,draw] (111) at (-4,2) {$111$};
    \draw (001) -- (101) node[midway,below] {$6$};
    \draw (011) -- (111) node[midway,above] {$6$};
    \draw (001) -- (011) node[midway,right] {$4$};
    \draw (101) -- (111) node[midway,left] {$5$};

    \draw (010) -- (011) node[midway,above] {$5$};
  \end{tikzpicture}
\end{center}

Any formal justification of this general pattern would be tedious and unreadable, so will will not attempt one.  The reader may, if they wish, treat all such diagrams as visual aids having no formal impact on the proofs.

\section{The detailed strategy for Theorem \ref{thm:positive}}
\label{sec:hard-bit}

The formalism of the previous section gives us very significant freedom to make radical changes to a linear datum $\Psi$.  However, to prove a general result, what we want is to find a sequence of operations that changes $\Psi$ as conservatively as possible, ideally giving back another datum of the same form with a small predictable change to some of the parameters.

Our task in this section then splits into two parts:
\begin{enumerate}[label=(\roman*)]
  \item to describe such a sequence of operations -- henceforth called a \emph{block} -- and analyze and verify the change it produces; and
  \item to show how to chain these blocks together to reach sufficiently arbitrary points in the parameter space.
\end{enumerate}

We will approach these tasks in reverse order.

\subsection{The effect of the block construction}
\label{subsec:block-strategy}

We again write $V = \FF_p^3$.  Suppose $X_1,\dots,X_6$ are six points in $\PP(V^\ast)$, corresponding to some system of six linear forms.  Suppose furthermore that $X_1,\dots,X_4$ are in general position, and that $X_5, X_6$ lie on some given line $\ell$ but $X_1,\dots,X_4$ do not lie on $\ell$.  Note we allow that, say, $X_1 X_2 X_5$ be collinear, or even that $X_5 = X_6$; in the latter case, $\ell$ forms part of the data of the set-up since it cannot be recovered from $X_1,\dots,X_6$.

We will now describe an operation that modifies this collection of points. Specifically, it will leave $X_1,\dots,X_4$ unchanged, and replace $X_5, X_6$ with two different points $X'_5, X'_6$ that both lie on the unchanged line $\ell$.

The points $X'_5, X'_6$ are constructed as follows.  Let $Y$ be the point at the intersection of the lines $X_1 X_5$ and $X_3 X_4$; then $X'_5$ is the intersection of $X_2 Y$ and $\ell$.  Similarly, letting $Z$ be the intersection of $X_2 X_6$ and $X_3 X_4$, the point $X'_6$ is the intersection of $X_1 Z$ and $\ell$.  This construction is shown in Figure \ref{fig:main-construction}.

\begin{figure}
  \begin{center}
    \begin{tikzpicture}[dot/.style={circle,inner sep=0pt,fill,minimum size=5pt},
      extended line/.style={shorten >=-#1,shorten <=-#1},
      extended line/.default=10cm,
      scale=0.8]
      \clip (-5,-5) rectangle (5,5);
      \node[dot,label={$X_5$}] (5) at (-2,4) {};
      \node[dot,label={$X_6$}] (6) at (3,3) {};
      \node[dot,label={$X_1$}] (1) at (-1,0) {};
      \node[dot,label={$X_2$}] (2) at (1,0) {};
      \node[dot,label={$X_3$}] (3) at (-3,-3) {};
      \node[dot,label={$X_4$}] (4) at (4,-4) {};

      \node[dot,label={$Y$}] (Y) at (intersection of 5--1 and 3--4) {};
      \node[dot,label={$Z$}] (Z) at (intersection of 6--2 and 3--4) {};

      \node[dot,label={$X'_5$}] (5p) at (intersection of Y--2 and 5--6) {};
      \node[dot,label={$X'_6$}] (6p) at (intersection of Z--1 and 5--6) {};

      \draw[thick, extended line] (5)--(6);
      \draw[thick, extended line] (3)--(4);

      \draw[thin, dotted, extended line] (5)--(Y);
      \draw[thin, dotted, extended line] (6)--(Z);
      \draw[thin, dotted, extended line] (Y)--(5p);
      \draw[thin, dotted, extended line] (Z)--(6p);
    \end{tikzpicture}
  \end{center}
  \caption{The block construction}
  \label{fig:main-construction}
\end{figure}
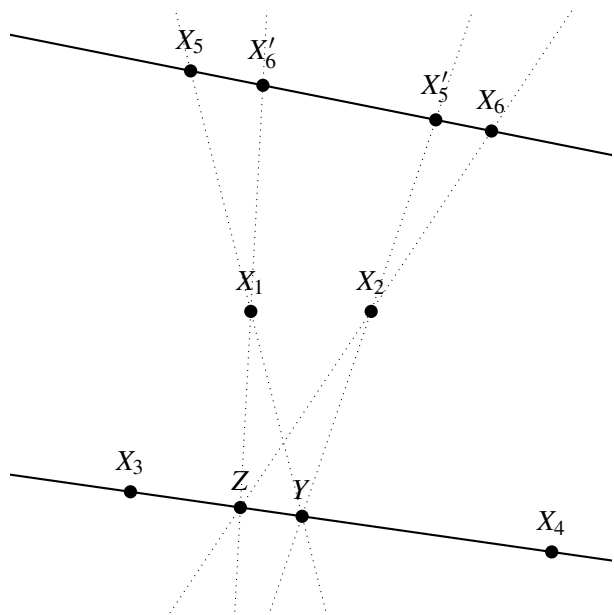
Given our hypotheses on $X_1,\dots,X_6$, this definition always makes sense.

We call this construction a \emph{block operation} $B_{1\to2}$, i.e.~$B_{1 \to 2}(X_1,\dots,X_6; \ell) = (X_1,\dots,X_4,X'_5,X'_6; \ell)$.  By exchanging the roles of $X_1,X_2,X_3,X_4$ we create a family of $12$ operations $B_{i\to j}$ for each pair $i,j \in [4]$, $i \ne j$.  Swapping $X_3$ and $X_4$ but leaving $X_1$ and $X_2$ the same gives the same construction, which accounts for the fact there are $12$ operations not $4! = 24$, and for the choice of notation.

In Section \ref{subsec:build-block}, we will implement a sequence of $\CS$, $\MERGE$ and $\TRIVIAL$ operations whose overall effect is equivalent to this $B_{1 \to 2}$ move.  That is, starting with a linear datum corresponding (indirectly) to forms $\phi_1,\dots,\phi_6$ and applying this sequence of operations, we obtain a linear datum corresponding to $B_{1\to 2}([\phi_1],\dots,[\phi_6])$.  We will not state this result precisely yet, as there are some technical subtleties to do with the case $[\phi_5] = [\phi_6]$, where the previous sentence does not even make sense and the datum has to be modified to encode the line $\ell$ as well as the points.\footnote{Handling this case correctly is an irritating source of complexity in the argument, but seems to be slightly less irritating than avoiding it.}

For the time being we will consider the operations $B_{i\to j}$ as a black box.  In order to prove Lemma \ref{lem:game}, we broadly need to show that some sequence of moves $B_{i \to j}$ takes the original $X_5,X_6$ to some final $X''_5,X''_6$, with the property that one of $X_i,X_j,X''_5$ or $X_i, X_j, X''_6$ are collinear for some choice of $1 \le i < j \le 4$, but the complementary triple are not collinear.

The second part of this is guaranteed by the following lemma, which shows that we never lose control of true complexity by applying $B_{1 \to 2}$ (and symmetrically $B_{i \to j}$ for other pairs $(i,j)$).
\begin{lemma}
  \label{lem:safe-block}
  Let $(X_1,\dots,X_6; \ell)$ be as above, and suppose:
  \begin{itemize}
    \item if $X_5 \ne X_6$, that $X_1,\dots,X_6$ do not lie on a \uppar{possibly degenerate} conic; or
    \item if $X_5 = X_6$, that $X_1,\dots,X_5$ do not lie on a \uppar{possibly degenerate} conic that is tangent to $\ell$ at $X_5$.
  \end{itemize}
  Then the same is true of $B_{1\to 2}(X_1,\dots,X_6; \ell)$.
\end{lemma}
Note that in the degenerate case, saying a degenerate conic consisting of two lines $\mu_1$,$\mu_2$ is ``tangent'' to $\ell$ translates algebraically to saying that $\mu_1,\mu_2,\ell$ are concurrent.
\begin{proof}
  For any point $Y$ on $\ell$, there is an unique (possibly degenerate) conic $C$ passing through $X_1,\dots,X_4$ and $Y$.  This conic meets $\ell$ again at precisely one other point, counting multiplicity, which we denote by $\tau(Y)$.  So, $\tau$ is an involution $\tau \colon \ell \to \ell$ of the points of $\ell$.

  It is clear without doing detailed calculations that $\tau$ is a birational map $\PP^1(\FF_p) \to \PP^1(\FF_p)$, and so it must be a M\"obius transformation (or one can check this directly).  Moreover, if we write $A_{ij}$ for $1 \le i < j \le 4$ for the intersection point of the lines $X_i X_j$ and $\ell$, then $\tau$ is characterized by
  \begin{equation}
    \label{eq:tau-props}
    \begin{aligned}
      \tau(A_{12}) &= A_{34} &
      \tau(A_{13}) &= A_{24} &
      \tau(A_{14}) &= A_{23} \\
      \tau(A_{34}) &= A_{12} &
      \tau(A_{24}) &= A_{13} &
      \tau(A_{23}) &= A_{14}
    \end{aligned}
  \end{equation}
  since in all of these cases the conic $C$ is degenerate and so $\tau(Y)$ can be found by inspection.

  Let $\sigma_{1\to2}$ be the map $\ell \to \ell$ sending $P \in \ell$ to $Q X_2 \cap \ell$ where $Q = X_1 P \cap X_3 X_4$ (i.e., how we obtained $X'_5$ from $X_5$ in the definition of $B_{1\to2}$).  Then $\sigma_{1\to 2}$ is also a M\"obius transformation $\ell \to \ell$, since it is the composition of two perspectivities, the first sending $\ell$ to $X_3 X_4$ via $X_1$ and the second sending $X_3 X_4$ to $\ell$ via $X_2$.  Moreover, it is characterized by
  \begin{equation}
    \label{eq:sigma-props}
    \begin{aligned}
      \sigma_{1\to2}(A_{12}) &= A_{12} &&&
    \sigma_{1\to2}(A_{13}) &= A_{23} \\
      \sigma_{1\to2}(A_{14}) &= A_{24} &&&
    \sigma_{1\to2}(A_{34}) &= A_{34}
    \end{aligned}
  \end{equation}
  as again the image point in these cases is immediate by inspection.  Let $\sigma_{i\to j}$ for $1 \le i \ne j \le 4$ be defined similarly by permuting the roles of $X_1,\dots,X_4$; so \eqref{eq:sigma-props} holds analogously for these under a suitable permutation of indices.  Note that $B_{1\to 2}(X_1,\dots,X_6; \ell) = (X_1,\dots,X_4,X'_5,X'_6; \ell)$ where $X'_5 = \sigma_{1 \to 2}(X_5)$ and $X'_6 = \sigma_{2 \to 1}(X_6)$.

  Now, the hypothesis on $X_1,\dots,X_6$ in the statement holds if and only if $\tau(X_5) \ne X_6$.  We wish to deduce that $\tau(X'_5) \ne X'_6$; equivalently, that $\tau \sigma_{1\to 2}(X_5) \ne \sigma_{2\to 1}(X_6)$.
  It would suffice to show that $\tau \sigma_{1\to 2} = \sigma_{2\to 1} \tau$ as M\"obius transformations $\ell \to \ell$.  However, this is immediate because
  \begin{align*}
    \tau \sigma_{1 \to 2} (A_{12}) = \sigma_{2 \to 1} \tau (A_{12}) &= A_{34} \\
    \tau \sigma_{1 \to 2} (A_{13}) = \sigma_{2 \to 1} \tau (A_{13}) &= A_{14} \\
    \tau \sigma_{1 \to 2} (A_{14}) = \sigma_{2 \to 1} \tau (A_{14}) &= A_{13}
  \end{align*}
  (using \eqref{eq:tau-props} and \eqref{eq:sigma-props}) and because it follows from our general position assumptions that $A_{13}$,$A_{14}$,$A_{34}$ are distinct points of $\ell$, and therefore uniquely determine a M\"obius transformation.
\end{proof}

This means that if we can find a sequence of operations $B_{i\to j}$ that takes $X_5$ to some $X''_5$ such that some triple $X_k, X_\ell, X''_5$ for $1 \le k < \ell \le 4$ is collinear, then the complementary triple \emph{cannot} be collinear as then $X_1,\dots,X_4,X''_5,X''_6$ would lie on a degenerate conic.  So, we can largely forget about $X_6$ in what follows and concentrate on the action of $B_{i\to j}$ on $X_5$, which corresponds to the M\"obius transformations $\sigma_{i\to j}$ defined in the proof of Lemma \ref{lem:safe-block}.

The following lemma explains how to take $X_5$ to an arbitrary point on $\ell$, relatively efficiently, using multiple transformations $\sigma_{i \to j}$.

\begin{lemma}
  \label{lem:euclid}
  We continue the notation from the proof of Lemma \ref{lem:safe-block}.  Suppose we identify $\ell$ with $\PP^1(\FF_p)$ \uppar{i.e., choose coordinates} by identifying $A_{12}$ with $\infty$, $A_{13}$ with $0$ and $A_{23}$ with $1$ \uppar{again noting these are guaranteed to be distinct}.  Every M\"obius transformation of $\ell$ now corresponds to a $2\times 2$ matrix in $\PGL_2(\FF_p)$.

  Then the following hold:
  \[
    \begin{aligned}
    \sigma_{4 \to 1} \sigma_{1 \to 2} \sigma_{2 \to 4} &= \begin{pmatrix} -1 & 1 \\ 0 & 1 \end{pmatrix} \,; &
      \sigma_{4 \to 2} \sigma_{2 \to 3} \sigma_{3 \to 4} &= \begin{pmatrix} 0 & 1 \\ 1 & 0 \end{pmatrix} \,; &
    \sigma_{4 \to 3} \sigma_{3 \to 1} \sigma_{1 \to 4} &= \begin{pmatrix} 1 & 0 \\ 1 & -1 \end{pmatrix} 
    \end{aligned}
  \]
and:
  \[
    \begin{aligned}
      \sigma_{3 \to 2} \sigma_{4 \to 3} \sigma_{3 \to 1} \sigma_{2 \to 3} \sigma_{3 \to 4} \sigma_{1 \to 3} &= \begin{pmatrix} 1 & 2 \\ 0 & 1 \end{pmatrix} \,; &&&
    \sigma_{3 \to 1} \sigma_{4 \to 3} \sigma_{3 \to 2} \sigma_{1 \to 3} \sigma_{3 \to 4} \sigma_{2 \to 3} &= \begin{pmatrix} 1 & -2 \\ 0 & 1 \end{pmatrix} \,; \\
      \sigma_{2 \to 3} \sigma_{4 \to 2} \sigma_{2 \to 1} \sigma_{3 \to 2} \sigma_{2 \to 4} \sigma_{1 \to 2} &= \begin{pmatrix} 1 & 0 \\ 2 & 1 \end{pmatrix} \,; &&&
    \sigma_{2 \to 1} \sigma_{4 \to 2} \sigma_{2 \to 3} \sigma_{1 \to 2} \sigma_{2 \to 4} \sigma_{3 \to 2} &= \begin{pmatrix} 1 & 0 \\ -2 & 1 \end{pmatrix} \,.
    \end{aligned}
  \]
  Consequently, the action of $\sigma_{i \to j}$ on $\PP^1(\FF_p)$ is transitive, and more specifically any point $[r:s] \in \PP^1(\FF_p)$, where $r,s \in \ZZ$, may be mapped to $[1:1]$ using a word of length $O(|r|+|s|)$ or $O(\log p)$ in $\sigma_{i \to j}$.
\end{lemma}
\begin{proof}
  The first three identities may be verified using only \eqref{eq:sigma-props} (and the corresponding statement for the other $\sigma_{i\to j}$) to deduce the action on $A_{12}$,$A_{13}$,$A_{23}$ (which correspond to $\infty,0,1$).

  For the next four, that approach does not appear to be sufficient.  Our strategy is just pick coordinates and compute explicitly.

  We can choose projective coordinates for $\PP(V^\ast)$ such that\footnote{
    To see this is possible, note that we can certainly choose a projective transformation sending $X_1$, $X_2$, $X_3$ to $[1:0:0]$, $[0:1:0]$ and $[0:0:1]$ as $X_1,X_2,X_3$ are not collinear. In these coordinates, $\ell = \{(x,y,z) \in \PP(V^\ast) \colon ax+by+cz = 0\}$ for some $a,b,c \in \FF_p \setminus \{0\}$ (as none of $X_1, X_2, X_3$ lie on $\ell$); by further rescaling we can ensure $a=b=c=1$.  This convention now differs from the one above -- which is somewhat more convenient -- by a further fixed change of coordinates.
  }
   $X_1 = [0:0:1]$, $X_2 = [1:0:1]$, $X_3 = [0:-1:1]$ and $\ell = \{ [x:y:z] \colon z = 0 \}$.  We write $X_4 = [a:b:1]$ for some $a,b \in \FF_p$: by our collinearity assumptions, this is possible, and furthermore $a \ne 0$, $b \ne 0$ and $a-b-1 \ne 0$.

  Using the information \eqref{eq:sigma-props}, we can compute the matrices of $\sigma_{i\to j}$ explicitly as
  \begin{align*}
    \sigma_{1 \to 2} &= \begin{pmatrix} a-b-1 & a \\ 0 & b \end{pmatrix} &
      \sigma_{1 \to 3} &= \begin{pmatrix} b & 0 \\ b & 1+b-a \end{pmatrix} &
    \sigma_{1 \to 4} &= \begin{pmatrix} a-1 & -a \\ b & -b-1 \end{pmatrix} \\
      \sigma_{2 \to 3} &= \begin{pmatrix} 0 & a \\ -b & a+b \end{pmatrix} &
        \sigma_{2 \to 4} &= \begin{pmatrix} a & 0 \\ b & 1 \end{pmatrix} &
    \sigma_{3 \to 4} &= \begin{pmatrix} -1 & a \\ 0 & b \end{pmatrix}
  \end{align*}
  with the other six following from the relation $\sigma_{i \to j} = \sigma_{j \to i}^{-1}$.  Verifying the remaining formulae is now just an exercise in multiplying (projective) matrices.

  By a modified version of Euclid's algorithm, any point $[r:s] \in \PP^1$ may be reduced to one of $\infty,0,1$ using the matrices $\begin{pmatrix} 1 & \pm 2 \\ 0 & 1 \end{pmatrix}$, $\begin{pmatrix} 1 & 0 \\ \pm 2 & 1 \end{pmatrix}$, in $O(|r|+|s|)$ steps (the worst case being something like $[r:1]$ for $r$ a large integer; the typical case is better). Alternatively, this can be done in $O(\log p)$ steps, since the Cayley graph on $\PSL_2(\FF_p)$ with these generators has diameter $O(\log p)$, as a corollary of celebrated results concerning expander graphs; see \cite{lubotzky}.  Finally, one of the first three matrices moves the end point to $[1:1]$, if necessary.  
\end{proof}

We should also verify that the values of $[r:s]$ representing the original point $X_5$ are not too large.
\begin{lemma}
  \label{lem:find-alpha}
  Suppose $X_i = [x_i:y_i:z_i]$  for $1 \le i \le 6$ where $x_i,y_i,z_i$ are integers, and $X_5 \ne X_6$.  Then in the coordinates on $\ell$ described in Lemma \ref{lem:euclid}, the point $X_5 \in \ell$ is identified with $[r:s]$ where $r,s$ are integers with $|r|,|s| = O(max\{|x_i|,|y_i|,|z_i| \colon 1 \le i \le 6\}^6)$.
\end{lemma}
\begin{proof}
  Write $|X_i X_j X_k|$ for the determinant of the $3 \times 3$ matrix whose columns are the vectors $(x_i,y_i,z_i)$, $(x_j,y_j,z_j)$, $(x_k,y_k,z_k)$.  Then we set
  \begin{align*}
    r &= |X_1 X_3 X_5| \, |X_2 X_5 X_6| \\
    s &= |X_1 X_2 X_5| \, |X_3 X_5 X_6| 
  \end{align*}
  and claim that $[r:s]$ are the coordinates of $X_5$ in the coordinate system of Lemma \ref{lem:euclid}.  One can verify that the definition of $[r:s]$ is invariant under rescaling any $(x_i,y_i,z_i)$, or under a change of projective coordinates on $\PP^2$.  Also, it is straightforward to verify this claim when $X_1 = [0:0:1]$, $X_2 = [1:0:1]$, $X_3 = [0:-1:1]$ and $\ell = \{ [x:y:z] \colon z = 0 \}$ (i.e.~in the coordinates on $\PP^2$ used in the proof of Lemma \ref{lem:euclid}).  It follows that the claim holds in general.
\end{proof}

We have one further technical issue to consider.  It will be convenient to construct the block that implements $B_{i \to j}$ only in cases where none of the triples $X_i,X_j,X_5$ or $X_i,X_j,X_6$ for $1 \le i < j \le 4$ is collinear.  This is not too onerous, as the first time we obtain a set of points where one of these triples is collinear, we can just stop and the conclusion of Lemma \ref{lem:game} will be satisfied.  However, for this to work we need to check that, when this happens, we are not in one of the degenerate cases where $X_5=X_6$.

We therefore check the following lemma.

\begin{lemma}
  \label{lem:no-final-coincide}
  Suppose $X_1,\dots,X_6$ and $\ell$ are as above and $X'_5$, $X'_6$ are the points returned by the block move $B_{1\to 2}(X_1,\dots,X_6; \ell)$.  Suppose also that no triple $X_i,X_j,X_5$ or $X_i,X_j,X_6$ for $1 \le i < j \le 4$ is collinear, but that some triple $X_i,X_j,X'_5$ or $X_i,X_j,X'_6$ is collinear.  Then $X'_5 \ne X'_6$.
\end{lemma}
\begin{proof}
  Suppose for contradiction that $X'_5 = X'_6 = A_{ij}$ for some $1 \le i < j \le 4$.  Then $X_5 = \sigma_{1\to 2}^{-1}(A_{ij})$ and $X_6 = \sigma_{2\to 1}^{-1}(A_{ij})$.  By \eqref{eq:sigma-props}, if either $i=2$ or $j=2$, or $i=3$ and $j=4$, then $X_5 = A_{i'j'}$ for some $1 \le i'<j'\le 4$, which is a contradiction.  Similarly, if either $i=1$ or $j=1$, or $i=3$ and $j=4$, then $X_6 = A_{i'j'}$ for some $1 \le i'<j'\le 4$, which is again a contradiction.  Since these cases exhaust all possible pairs $i,j$, the result follows.
\end{proof}

\subsection{Implementing a block move}
\label{subsec:build-block}

In this subsection we describe a sequence of $\CS$, $\MERGE$ and $\TRIVIAL$ operations that have the effect of a block transformation $B_{1 \to 2}$.  This is the last but most central ingredient in the proof of Theorem \ref{thm:positive}.

In the course of the argument in Section \ref{subsec:block-strategy}, we may need to consider intermediate configurations $(X_1,\dots, X_6) \in \PP^2$ for which $X_5 = X_6$.  Typically we do not expect this case to arise, but it would be onerous to try to avoid it in general.  Also, it is not true that in such cases we are immediately done by some easy Cauchy--Schwarz technique as in Section \ref{sec:63}: it appears this degeneracy is not one we can use to our advantage.

To handle this, we need to build more slack into our linear datum.  We again set $V = \FF_p^3$.

\begin{definition}
  \label{def:augmented}
  Let $\phi_1,\dots,\phi_6 \colon V \to \FF_p$ be non-zero linear forms, with $[\phi_1],\dots,[\phi_4]$ in general position, and let $\ell$ be a subspace of dimension $2$ in $V^\ast$ containing $\phi_5$, $\phi_6$ but none of $\phi_1,\dots,\phi_4$.  An \emph{augmented datum representing $(\phi_1,\dots,\phi_6;\ell)$} is a linear datum $\Psi = \left(V \oplus \FF_p, (W_i)_{i\in I}, (\psi_i)_{i\in I}\right)$ where $I =[6]$, $W_i = \FF_p$ for each $1 \le i \le 4$ and $W_5 = W_6 = \FF_p^2$, and $\psi_i \colon V \oplus \FF_p \to W_i$ satisfy:
  \begin{itemize}
    \item for $1 \le i \le 4$ and any $(v,t) \in V \oplus \FF_p$ we have $\psi_i(v, t) = \phi_i(v)$;
    \item for $i = 5, 6$ and any $(v, t) \in V \oplus \FF_p$ we have $\psi_i(v, t) = (\phi_i(v), t + \chi_i(v))$ for some $\chi_i \in \ell$.
  \end{itemize}

  Also, the \emph{standard datum representing $(\phi_1,\dots,\phi_6)$} is just $\Psi' = \left(V, (W'_i)_{i \in I}, (\phi_i)_{i \in I}\right)$ where $I = [6]$ and $W'_i = \FF_p$ for each $i$, as in Lemma \ref{lem:game}.
\end{definition}
There are many roughly equally cryptic ways to phrase this rigorously.  Geometrically, what has happened is that we have embedded our projective plane $\PP^2 = \PP(V^\ast)$ in a three-dimensional space $\PP^3 = \PP((V \oplus \FF_p)^\ast)$, and each of $\psi_1,\dots,\psi_4$ corresponds to the respective point $[\phi_1],\dots,[\phi_4]$ in the embedded copy of $\PP^2$. Meanwhile, $\psi_5$ and $\psi_6$ correspond to lines in $\PP^3$ whose intersections with the embedded $\PP^2$ are $[\phi_5],[\phi_6]$, and whose canonical projections onto the embedded $\PP^2$ are both (contained in, but secretly equal to) the line $\ell$.

In the case $[\phi_5] \ne [\phi_6]$, we don't really need this extra dimension but it does no harm, as the following lemma will show.  When $[\phi_5] = [\phi_6]$, the situation has genuinely changed because the augmented datum retains information about $\ell$ whereas the standard one would not.

\begin{lemma}
  \label{lem:make-augmented}
  Let $\phi_1,\dots,\phi_6 \colon V \to \FF_p$ be linear forms as in Definition \ref{def:augmented}, and suppose further that $[\phi_5] \ne [\phi_6]$ and $\ell = \spn(\phi_5,\phi_6)$.  Let $\Psi_1 = (V \oplus \FF_p, (W_i)_{i \in I}, (\psi_i)_{i \in I})$ be an augmented datum representing $(\phi_1,\dots,\phi_6;\ell)$ and let $\Psi_2 = (V, (W'_i)_{i \in I}, (\phi_i)_{i \in I})$ be the standard datum representing $(\phi_1,\dots,\phi_6)$.  Then $\Psi_1$ dominates $\Psi_2$ respecting $(j,j)$ and $\Psi_2$ dominates $\Psi_1$ respecting $(j, j)$, for any $1 \le j \le 4$, and with exponent $1$ in each case.
\end{lemma}
\begin{proof}
  Both directions are by $\TRIVIAL$ steps (i.e.~Proposition \ref{prop:trivial-dominate}).  First we show that $\Psi_1$ dominates $\Psi_2$ trivially (respecting $1 \le j \le 4$).  Indeed, we may consider the surjective maps $\theta \colon V \oplus \FF_p \to V$ given by $\theta(v,t) = v$, and $\sigma_i \colon W_i \to W'_i$ given by the identity if $1 \le i \le 4$ and $\sigma_5(x,y) = \sigma_6(x,y) = x$.  It is immediate from our hypotheses that $\phi_i \circ \theta = \sigma_i \circ \psi_i$ for each $i$, and the claim follows.

  To show $\Psi_2$ dominates $\Psi_1$ trivially (respecting $1 \le j \le 4$), we consider an injective map $\imath \colon V \to V \oplus \FF_p$ given by $\imath(v) = (v, \mu(v))$ for some $\mu \in V^\ast$ which will have to be chosen carefully, together with $\tau_i \colon W'_i \to W_i$ given by the identity if $1 \le i \le 4$ and to be specified when $i=5,6$.  It suffices to show that $\psi_i \circ \imath = \tau_i \circ \phi_i$ for each $1 \le i \le 6$, under suitable choices.  Note that this is already immediate for $1 \le i \le 4$, given our hypotheses.  For $i = 5,6$, we need precisely that
  \begin{align*}
    \psi_i(v, \mu(v)) = (\phi_i(v), \mu(v) + \chi_i(v)) = \tau_i(\phi_i(v))
  \end{align*}
  for any $v \in V$.  If $\mu + \chi_i \in \spn(\phi_i)$ for $i=5,6$ as elements of $V^\ast$, so $\mu + \chi_i = \gamma_i \phi_i$ for some $\gamma_5,\gamma_6 \in \FF_p$, we could define $\tau_i(x) = (x, \gamma_i x)$ and the equation would be satisfied.
  
  We check that this holds for an appropriate choice of $\mu$.  Because $[\phi_5] \ne [\phi_6]$, $\phi_5,\phi_6$ is a basis for $\ell$ and so we may write
  \[
    \chi_5 - \chi_6 = \alpha \phi_5 + \beta \phi_6
  \]
  for some $\alpha,\beta \in \FF_p$.  Define $\mu = \alpha \phi_5 - \chi_5 = -\beta \phi_6 - \chi_6$; hence, $\chi_5 + \mu \in \spn(\phi_5)$ and $\chi_6 + \mu \in \spn(\phi_6)$ as required.
\end{proof}

Again it may be instructive to think about this geometrically.  In the first part, we used the canonical embedding of $\PP^2$ into $\PP^3$ discussed above to get our morphism of linear data.  In the second part, we chose a particular non-standard projection $\PP^3 \to \PP^2$ that collapses the line corresponding to $\psi_5$ onto $[\phi_5]$ and that corresponding to $\psi_6$ onto $[\phi_6]$.

Finally, we can state a lemma which is the workhorse of the whole argument.
\begin{lemma}
  \label{lem:key-lemma}
  Let $\phi_1,\dots,\phi_6$ and $\ell$ be as in Definition \ref{def:augmented}, and let $\Psi$ be an augmented datum representing $(\phi_1,\dots,\phi_6; \ell)$.  Also, let $\phi'_5,\phi'_6 \colon V \to \FF_p$ be linear forms such that
  \[
    B_{1 \to 2}([\phi_1],\dots,[\phi_6]; \ell) = ([\phi_1],\dots,[\phi_4], [\phi'_5], [\phi'_6]; \ell) \, .
  \]
  Then there exists an augmented datum $\Psi'$ representing $(\phi_1,\dots,\phi_4,\phi'_5,\phi'_6;\ell)$, such that $\Psi'$ dominates $\Psi$ respecting $(j,j)$ for any $1 \le j \le 4$, and with exponent $1/16$.
\end{lemma}

The argument has roughly two phases.  In the first phase, our goal is to build a datum corresponding to the following graph of vector spaces over $V \oplus \FF_p$:
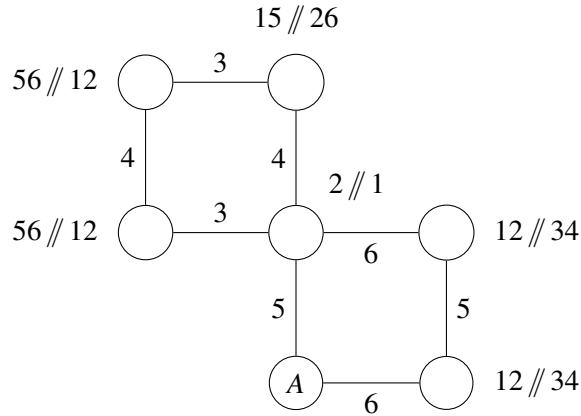
\begin{center}
  \begin{tikzpicture}
    \node[circle,draw] (0000) at (0,0) {$A$};
    \node[circle,draw] (1000) at (2,0) {$\phantom{B}$};
    \node[right] at ($(1000) + (0.5, 0)$) {$12 \sslash 34$};
    \node[circle,draw] (0100) at (0,2) {$\phantom{C}$};
    \node[above right] at ($(0100) + (0.3, 0.3)$) {$2 \sslash 1$};
    \node[circle,draw] (1100) at (2,2) {$\phantom{D}$};
    \node[right] at ($(1100) + (0.5, 0)$) {$12 \sslash 34$};
    \draw (0000) -- (1000) node[midway,below] {$6$};
    \draw (0100) -- (1100) node[midway,below] {$6$};
    \draw (0000) -- (0100) node[midway,left] {$5$};
    \draw (1000) -- (1100) node[midway,right] {$5$};

    \node[circle,draw] (0110) at (-2,2) {$\phantom{E}$};
    \node[left] at ($(0110) + (-0.5, 0)$) {$56 \sslash 12$};

    \draw (0100) -- (0110) node[midway,above] {$3$};

    \node[circle,draw] (0101) at (0,4) {$\phantom{G}$};
    \node[above] at ($(0101) + (0.0, 0.5)$) {$15 \sslash 26$};

    \node[circle,draw] (0111) at (-2,4) {$\phantom{F}$};
    \node[left] at ($(0111) + (-0.5, 0)$) {$56 \sslash 12$};

    \draw (0101) -- (0111) node[midway,above] {$3$};

    \draw (0100) -- (0101) node[midway,left] {$4$};
    \draw (0110) -- (0111) node[midway,left] {$4$};
  \end{tikzpicture}
\end{center}
Here the numbers next to the vertices denote two classes corresponding to those indices that get merged into $\psi'_5$ and those that get merged into $\psi'_6$ respectively.  The indices at vertex $A$ will turn into $\psi'_1,\dots,\psi'_4$.

It is not possible to construct this graph directly using $\CS$ steps, so we have to build a larger graph using $\CS$ steps and then prune it back using $\MERGE$ and $\TRIVIAL$ steps.

In the second phase, we need to apply a carefully chosen $\TRIVIAL$ operation to reduce this to a system $\Psi'$ defined on a single copy of $V \oplus \FF_p$.

\begin{proof}[Proof of Lemma \ref{lem:key-lemma}]
  We abbreviate $V \oplus \FF_p$ to $V'$.  Beginning with the augmented datum $\Psi = \left(V', (W_i)_{i \in [6]}, (\psi_i)_{i \in [6]}\right)$, we first apply $\CS_6$:
  \begin{center}
    \begin{tikzpicture}
      \node[circle,draw] (0) at (0,0) {$0$};
      \node[circle,draw] (1) at (2,0) {$1$};
      \draw (0) -- (1) node[midway,below] {$6$};
    \end{tikzpicture}
  \end{center}
  and then $\MERGE_{\{5_0,5_1\} \mapsto R}$ followed by $\CS_R$ to get:
  
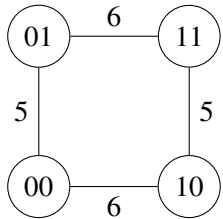
\begin{center}
    \begin{tikzpicture}
      \node[circle,draw] (00) at (0,0) {$00$};
      \node[circle,draw] (10) at (2,0) {$10$};
      \node[circle,draw] (01) at (0,2) {$01$};
      \node[circle,draw] (11) at (2,2) {$11$};
      \draw (00) -- (10) node[midway,below] {$6$};
      \draw (01) -- (11) node[midway,above] {$6$};
      \draw (00) -- (01) node[midway,left] {$5$};
      \draw (10) -- (11) node[midway,right] {$5$};
    \end{tikzpicture}
  \end{center}
  Now we do $\CS_{3_{01}}$ to get:
  
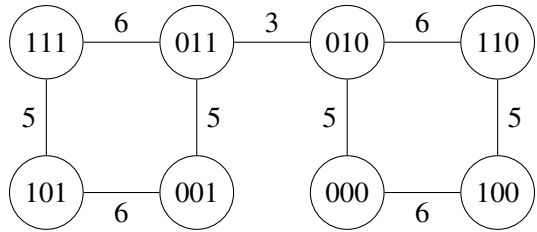
\begin{center}
    \begin{tikzpicture}
      \node[circle,draw] (000) at (0,0) {$000$};
      \node[circle,draw] (100) at (2,0) {$100$};
      \node[circle,draw] (010) at (0,2) {$010$};
      \node[circle,draw] (110) at (2,2) {$110$};
      \draw (000) -- (100) node[midway,below] {$6$};
      \draw (010) -- (110) node[midway,above] {$6$};
      \draw (000) -- (010) node[midway,left] {$5$};
      \draw (100) -- (110) node[midway,right] {$5$};

      \node[circle,draw] (001) at (-2,0) {$001$};
      \node[circle,draw] (101) at (-4,0) {$101$};
      \node[circle,draw] (011) at (-2,2) {$011$};
      \node[circle,draw] (111) at (-4,2) {$111$};
      \draw (001) -- (101) node[midway,below] {$6$};
      \draw (011) -- (111) node[midway,above] {$6$};
      \draw (001) -- (011) node[midway,right] {$5$};
      \draw (101) -- (111) node[midway,left] {$5$};

      \draw (010) -- (011) node[midway,above] {$3$};
    \end{tikzpicture}
  \end{center}
  followed by $\MERGE_{\{4_{010},4_{011}\} \mapsto R}$ and then $\CS_R$ to get:
  
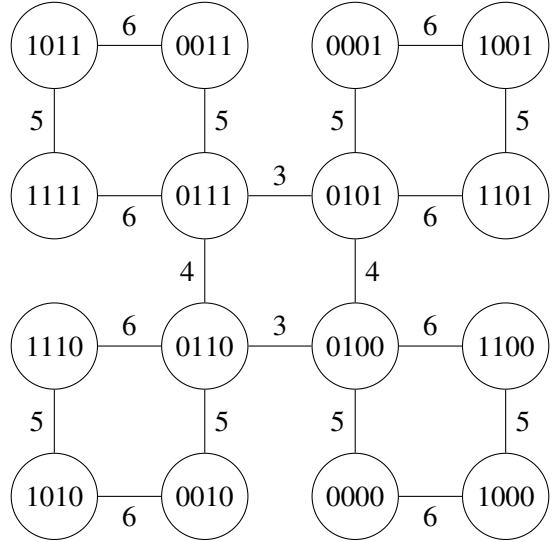
\begin{center}
    \begin{tikzpicture}
      \node[circle,draw] (0000) at (0,0) {$0000$};
      \node[circle,draw] (1000) at (2,0) {$1000$};
      \node[circle,draw] (0100) at (0,2) {$0100$};
      \node[circle,draw] (1100) at (2,2) {$1100$};
      \draw (0000) -- (1000) node[midway,below] {$6$};
      \draw (0100) -- (1100) node[midway,above] {$6$};
      \draw (0000) -- (0100) node[midway,left] {$5$};
      \draw (1000) -- (1100) node[midway,right] {$5$};

      \node[circle,draw] (0010) at (-2,0) {$0010$};
      \node[circle,draw] (1010) at (-4,0) {$1010$};
      \node[circle,draw] (0110) at (-2,2) {$0110$};
      \node[circle,draw] (1110) at (-4,2) {$1110$};
      \draw (0010) -- (1010) node[midway,below] {$6$};
      \draw (0110) -- (1110) node[midway,above] {$6$};
      \draw (0010) -- (0110) node[midway,right] {$5$};
      \draw (1010) -- (1110) node[midway,left] {$5$};

      \draw (0100) -- (0110) node[midway,above] {$3$};

      \node[circle,draw] (0001) at (0,6) {$0001$};
      \node[circle,draw] (1001) at (2,6) {$1001$};
      \node[circle,draw] (0101) at (0,4) {$0101$};
      \node[circle,draw] (1101) at (2,4) {$1101$};
      \draw (0001) -- (1001) node[midway,above] {$6$};
      \draw (0101) -- (1101) node[midway,below] {$6$};
      \draw (0001) -- (0101) node[midway,left] {$5$};
      \draw (1001) -- (1101) node[midway,right] {$5$};

      \node[circle,draw] (0011) at (-2,6) {$0011$};
      \node[circle,draw] (1011) at (-4,6) {$1011$};
      \node[circle,draw] (0111) at (-2,4) {$0111$};
      \node[circle,draw] (1111) at (-4,4) {$1111$};
      \draw (0011) -- (1011) node[midway,above] {$6$};
      \draw (0111) -- (1111) node[midway,below] {$6$};
      \draw (0011) -- (0111) node[midway,right] {$5$};
      \draw (1011) -- (1111) node[midway,left] {$5$};

      \draw (0101) -- (0111) node[midway,above] {$3$};

      \draw (0100) -- (0101) node[midway,right] {$4$};
      \draw (0110) -- (0111) node[midway,left] {$4$};
    \end{tikzpicture}
  \end{center}

  Note that in these diagrams, the set of indices of the corresponding linear data are $i_\omega$ where $i \in \{1,\dots,6\}$, $\omega \in \{0,1\}^k$ and there is no edge labelled $i$ incident to vertex $\omega$.  In other words, there is a surviving linear form attached to each vertex $\omega$ (which has not been Cauchy--Schwarzed away) for each $i \in \{1,\dots,6\}$ which is not a vertex label at $\omega$.

  Denote this last datum by $\Psi_1 = \left(\cV, \big(W_i^{(1)}\big)_{i \in I_1}, \big(\psi_i^{(1)}\big)_{i \in I_1}\right)$.  Explicitly: $\cV$ is the subspace of $V'^{\{0,1\}^4}$ determined by the above graph of vector spaces; the index set is
  \[
    I_1 = \big\{ 1_\omega, 2_\omega \colon \omega \in \{0,1\}^4 \big\} \cup \big\{ 3_\omega, 4_\omega \colon \omega \in \{0,1\}^4 \setminus \{ 0100, 0101, 0110, 0111 \} \big\}\, ;
  \]
  the vector spaces $W^{(1)}_i$ for $i \in I_1$ are all just $\FF_p$; and $\psi^{(1)}_i \colon \cV \to W^{(1)}_i$ are given by
  \[
    \psi^{(1)}_{r_\omega}(v'_{0000},\dots,v'_{1111}) = \psi_r(v'_\omega) \, .
  \]
  Our next task is to prune back all of the $(5,6)$ squares apart from the bottom right one using $\MERGE$ and $\TRIVIAL$ steps.  We will need the following standard linear algebra fact.
  \begin{lemma}
    \label{lem:easy-prune}
    Let $L, L_1, L_2$ be vector spaces and $s_i \colon L \to L_i$ for $i=1,2$ be linear maps.  Then there exist maps $\fs_1,\fs_2 \colon L \to L$ such that
    \begin{itemize}
      \item $s_i \circ \fs_i = s_i$ for $i=1,2$;
      \item $\ker \fs_i = \ker s_i$ for $i=1,2$; and
      \item the maps $\fs_1$, $\fs_2$ commute.
    \end{itemize}
  \end{lemma}
  \begin{proof}
    Pick a basis for $\ker s_1 \cap \ker s_2$, and extend it separately to a basis for $\ker s_1$ and $\ker s_2$; merging these gives a basis for $\ker s_1 + \ker s_2$.  Finally extend this to a basis for $L$.  This gives a direct sum decomposition $L = K_{00} \oplus K_{01} \oplus K_{10} \oplus K_{11}$ where $\ker s_1 = K_{00} + K_{01}$ and $\ker s_2 = K_{00} + K_{10}$.  Let $\fs_1(x_{00},x_{01},x_{10},x_{11}) = (0,0,x_{10},x_{11})$ and $\fs_2(x_{00},x_{01},x_{10},x_{11}) = (0, x_{01}, 0, x_{11})$.  It is clear these maps have the desired properties.
  \end{proof}

  Consider e.g.~the bottom left square consisting of ${0010}$, ${0110}$, ${1010}$, ${1110}$.  We now merge the indices $i_{0010}$ for $1 \le i \le 4$ into a single index $R$, and all eight indices $i_{1010}$, $i_{1110}$ for $1 \le i \le 4$ into a single index $S$; that is, we apply
  \[
    \MERGE_{\{1_{0010},  2_{0010},   3_{0010},   4_{0010} \} \mapsto R,\ \{ 1_{1010},  2_{1010},   3_{1010},   4_{1010}, 1_{1110},  2_{1110},   3_{1110},   4_{1110}  \} \mapsto S}
  \]
to $\Psi_1$ to obtain the merged datum $\Psi_2 = \left(\cV, \big(W^{(2)}_i\big)_{i \in I_2}, \big(\psi^{(2)}_i\big)_{i \in I_2} \right)$. 

  Let $\cV'$ denote the vector space associated to the following graph of vector spaces on $V'$:
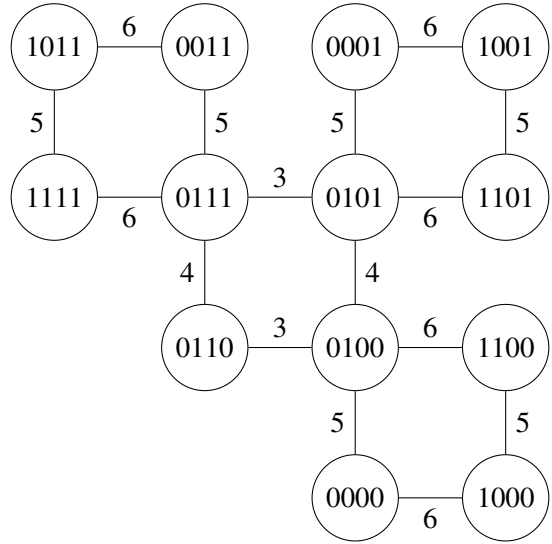
  \begin{center}
    \begin{tikzpicture}
      \node[circle,draw] (0000) at (0,0) {$0000$};
      \node[circle,draw] (1000) at (2,0) {$1000$};
      \node[circle,draw] (0100) at (0,2) {$0100$};
      \node[circle,draw] (1100) at (2,2) {$1100$};
      \draw (0000) -- (1000) node[midway,below] {$6$};
      \draw (0100) -- (1100) node[midway,above] {$6$};
      \draw (0000) -- (0100) node[midway,left] {$5$};
      \draw (1000) -- (1100) node[midway,right] {$5$};

      \node[circle,draw] (0110) at (-2,2) {$0110$};

      \draw (0100) -- (0110) node[midway,above] {$3$};

      \node[circle,draw] (0001) at (0,6) {$0001$};
      \node[circle,draw] (1001) at (2,6) {$1001$};
      \node[circle,draw] (0101) at (0,4) {$0101$};
      \node[circle,draw] (1101) at (2,4) {$1101$};
      \draw (0001) -- (1001) node[midway,above] {$6$};
      \draw (0101) -- (1101) node[midway,below] {$6$};
      \draw (0001) -- (0101) node[midway,left] {$5$};
      \draw (1001) -- (1101) node[midway,right] {$5$};

      \node[circle,draw] (0011) at (-2,6) {$0011$};
      \node[circle,draw] (1011) at (-4,6) {$1011$};
      \node[circle,draw] (0111) at (-2,4) {$0111$};
      \node[circle,draw] (1111) at (-4,4) {$1111$};
      \draw (0011) -- (1011) node[midway,above] {$6$};
      \draw (0111) -- (1111) node[midway,below] {$6$};
      \draw (0011) -- (0111) node[midway,right] {$5$};
      \draw (1011) -- (1111) node[midway,left] {$5$};

      \draw (0101) -- (0111) node[midway,above] {$3$};

      \draw (0100) -- (0101) node[midway,right] {$4$};
      \draw (0110) -- (0111) node[midway,left] {$4$};
    \end{tikzpicture}
  \end{center}
  and define a datum $\Psi_3 = \left(\cV', \big(W_i^{(3)}\big)_{i \in I_3}, \big(\psi^{(3)}_i\big)_{i \in I_3}\right)$ where:
  \begin{itemize}
    \item $I_3$ is the same as $I_2$;
    \item $W^{(3)}_i = W^{(2)}_i = \FF_p$ and $\psi^{(3)}_i = \psi^{(2)}_i = \psi^{(1)}_i$ for every $i \ne R,S$;
    \item $W_R = W_S = \FF_p^2$ and 
      \begin{align*}
        \psi^{(3)}_R\big((v'_\omega)_{\omega \in \{0,1\}^4 \setminus \{0010, 1010, 1110\} }\big) &= \psi_5(v'_{0110}) \\
        \psi^{(3)}_S\big((v'_\omega)_{\omega \in \{0,1\}^4 \setminus \{0010, 1010, 1110\} }\big) &= \psi_6(v'_{0110}) \, .
      \end{align*}
  \end{itemize}
  (It is not very important, but these are all surjective, as can be seen by considering the image of the diagonal embedding $V' \to \cV'$.)

  We claim that $\Psi_2$ is dominated trivially by the ``pruned'' datum $\Psi_3$.
  To justify this, we first apply Lemma \ref{lem:easy-prune} to $V'$, $\psi_5$ and $\psi_6$ to obtain maps $\fs_1, \fs_2 \colon V' \to V'$.  We can then define an injection $\imath \colon \cV' \to \cV$ by
  \[
    \imath\big((v'_\omega)_{\omega \in \{0,1\}^4 \setminus \{0010, 1010, 1110\} }\big) = \omega \mapsto \begin{cases} v'_\omega &\colon \omega \ne 0010, 1010, 1110 \\
      \fs_1(v'_{0110}) &\colon \omega = 0010 \\
      \fs_2(v'_{0110}) &\colon \omega = 1110 \\
    \fs_1 \fs_2(v'_{0110}) &\colon \omega = 1010 \end{cases} \, .
  \]
  For this to make sense, we need the compatibility conditions associated to the graph for $\cV$ to hold. In particular we need
  \begin{align*}
    \psi_5(v'_{0110}) &= \psi_5(\fs_1(v'_{0110})) \\
    \psi_5(\fs_2(v'_{0110})) &= \psi_5(\fs_1 \fs_2(v'_{0110})) \\
    \psi_6(v'_{0110}) &= \psi_6(\fs_2(v'_{0110})) \\
    \psi_6(\fs_1(v'_{0110})) &= \psi_6(\fs_1 \fs_2(v'_{0110})) 
  \end{align*}
  and indeed these follow from the properties of $\fs_1$ and $\fs_2$.  The remaining compatibility conditions are inherited from $\cV'$.

  We now give the $\TRIVIAL$ step explicitly.  For $i \ne R,S$ the map $\sigma_i \colon W^{(3)}_i \to W^{(2)}_i$ is just the identity.  For $R$ and $S$, consider that
  \begin{align*}
    \ker \left(\psi^{(2)}_R \circ \imath\right) &\supseteq \left\{ (v'_{0000},\dots,v'_{1111}) \in \cV' \colon \fs_1(v'_{0110}) = 0 \right\} \\
      &= \left\{ (v'_{0000},\dots,v'_{1111}) \in \cV' \colon \psi_5(v'_{0110}) = 0 \right\}  = \ker \psi^{(3)}_R \\
    \ker \left(\psi^{(2)}_S \circ \imath\right) &\supseteq \left\{ (v'_{0000},\dots,v'_{1111}) \in \cV' \colon \fs_2(v'_{0110}) = 0 \right\} \\
      &= \left\{ (v'_{0000},\dots,v'_{1111}) \in \cV' \colon \psi_6(v'_{0110}) = 0 \right\} = \ker \psi^{(3)}_S 
  \end{align*}
  as the original constituent forms of $\psi^{(2)}_R$ depend only on $v'_{0010}$ and those of $\psi^{(2)}_S$ on $(v'_{1010}, v'_{1110})$ in $\cV$.
  It follows that there exist unique linear maps $\sigma_R \colon W_R^{(3)} \to W_R^{(2)}$ and $\sigma_S \colon W_S^{(3)} \to W_S^{(2)}$ such that $\psi^{(2)}_R \circ \imath = \sigma_R \circ \psi_R^{(3)}$ and $\psi^{(2)}_S \circ \imath = \sigma_S \circ \psi_S^{(3)}$, respectively.  Hence the conditions of Proposition \ref{prop:trivial-dominate} are satisfied and $\Psi_3$ dominates $\Psi_2$ trivially.

  It is natural to relabel $R$ as $5_{0110}$ and $S$ as $6_{0110}$ in $\Psi_3$, as these indices now behave exactly like copies of $\psi_5$ and $\psi_6$ respectively associated to the vertex $0110$.

  We can summarize the preceding argument, which took us from the $16$-vertex graph to the $13$-vertex graph above, informally as follows.  On the dual side, we can say that we built a projection map $\imath^\ast \colon \cV^\ast \to \cV'^\ast$ that maps each of $\big(\ker \psi^{(2)}_{i_{0010}}\big)^\perp$ into $\big(\ker \psi^{(3)}_R\big)^\perp$ and each of $\big(\ker \psi^{(2)}_{i_{1010}}\big)^\perp$ or $\big(\ker \psi^{(2)}_{i_{1110}}\big)^\perp$ into $\big(\ker \psi^{(3)}_S\big)^\perp$ (for $i \in \{1,\dots,4\}$).  Moreover, $\imath^\ast$ was constructed by projecting out the spare coordinates at the vertices $0010$, $1010$, $11110$ using $\fs_1^\ast$ along vertical edges and $\fs_2^\ast$ along horizontal edges, and we checked this made sense.  This is what is meant by the following further annotated diagram (we will see more of this kind below).

  
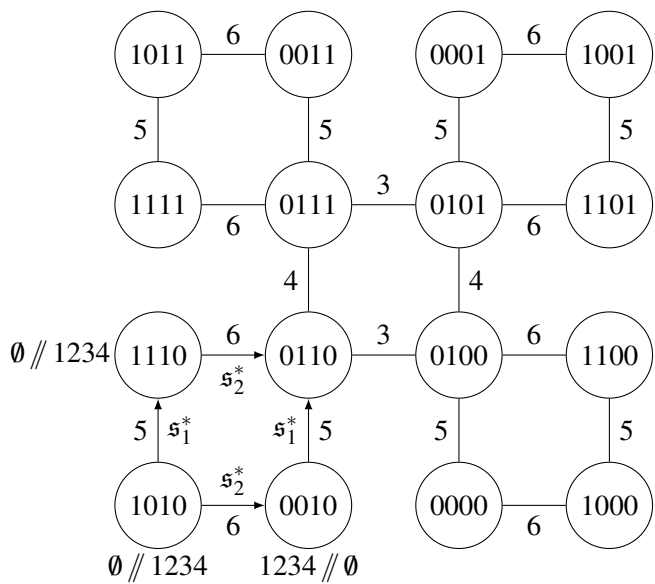
\begin{center}
    \begin{tikzpicture}
      \node[circle,draw] (0000) at (0,0) {$0000$};
      \node[circle,draw] (1000) at (2,0) {$1000$};
      \node[circle,draw] (0100) at (0,2) {$0100$};
      \node[circle,draw] (1100) at (2,2) {$1100$};
      \draw (0000) -- (1000) node[midway,below] {$6$};
      \draw (0100) -- (1100) node[midway,above] {$6$};
      \draw (0000) -- (0100) node[midway,left] {$5$};
      \draw (1000) -- (1100) node[midway,right] {$5$};

      \node[circle,draw] (0010) at (-2,0) {$0010$};
      \node[circle,draw] (1010) at (-4,0) {$1010$};
      \node[circle,draw] (0110) at (-2,2) {$0110$};
      \node[circle,draw] (1110) at (-4,2) {$1110$};
      \draw[-latex] (1010) -- (0010) node[midway,below] {$6$} node[midway,above] {$\fs_2^\ast$};
      \draw[-latex] (1110) -- (0110) node[midway,above] {$6$} node[midway,below] {$\fs_2^\ast$};
      \draw[-latex] (0010) -- (0110) node[midway,right] {$5$} node[midway,left] {$\fs_1^\ast$};
      \draw[-latex] (1010) -- (1110) node[midway,left] {$5$} node[midway,right] {$\fs_1^\ast$};
      \node[below] at ($(0010) + (0, -0.5)$) {$1234 \sslash \emptyset$};
      \node[left] at ($(1110) + (-0.5, 0)$) {$\emptyset \sslash 1234 $};
      \node[below] at ($(1010) + (0, -0.5)$) {$\emptyset \sslash 1234 $};

      \draw (0100) -- (0110) node[midway,above] {$3$};

      \node[circle,draw] (0001) at (0,6) {$0001$};
      \node[circle,draw] (1001) at (2,6) {$1001$};
      \node[circle,draw] (0101) at (0,4) {$0101$};
      \node[circle,draw] (1101) at (2,4) {$1101$};
      \draw (0001) -- (1001) node[midway,above] {$6$};
      \draw (0101) -- (1101) node[midway,below] {$6$};
      \draw (0001) -- (0101) node[midway,left] {$5$};
      \draw (1001) -- (1101) node[midway,right] {$5$};

      \node[circle,draw] (0011) at (-2,6) {$0011$};
      \node[circle,draw] (1011) at (-4,6) {$1011$};
      \node[circle,draw] (0111) at (-2,4) {$0111$};
      \node[circle,draw] (1111) at (-4,4) {$1111$};
      \draw (0011) -- (1011) node[midway,above] {$6$};
      \draw (0111) -- (1111) node[midway,below] {$6$};
      \draw (0011) -- (0111) node[midway,right] {$5$};
      \draw (1011) -- (1111) node[midway,left] {$5$};

      \draw (0101) -- (0111) node[midway,above] {$3$};

      \draw (0100) -- (0101) node[midway,right] {$4$};
      \draw (0110) -- (0111) node[midway,left] {$4$};
    \end{tikzpicture}
  \end{center}

  We then repeat these same steps on the top left and top right corners.  The result is a datum $\Psi_4 = \left(\cV'', \big(\psi^{(4)}_i\big)_{i \in I_4}, \big(W_i^{(4)}\big)_{i \in I_4} \right)$ corresponding to the $7$-vertex configuration
  
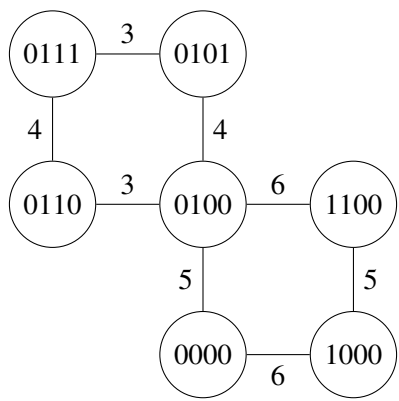
\begin{center}
    \begin{tikzpicture}
      \node[circle,draw] (0000) at (0,0) {$0000$};
      \node[circle,draw] (1000) at (2,0) {$1000$};
      \node[circle,draw] (0100) at (0,2) {$0100$};
      \node[circle,draw] (1100) at (2,2) {$1100$};
      \draw (0000) -- (1000) node[midway,below] {$6$};
      \draw (0100) -- (1100) node[midway,above] {$6$};
      \draw (0000) -- (0100) node[midway,left] {$5$};
      \draw (1000) -- (1100) node[midway,right] {$5$};

      \node[circle,draw] (0110) at (-2,2) {$0110$};

      \draw (0100) -- (0110) node[midway,above] {$3$};

      \node[circle,draw] (0101) at (0,4) {$0101$};

      \node[circle,draw] (0111) at (-2,4) {$0111$};

      \draw (0101) -- (0111) node[midway,above] {$3$};

      \draw (0100) -- (0101) node[midway,right] {$4$};
      \draw (0110) -- (0111) node[midway,left] {$4$};
    \end{tikzpicture}
  \end{center}
  and which dominates $\Psi_3$, and hence $\Psi$, respecting $(r_{0000}, r)$ for $1 \le r \le 4$, with exponent $1$.  Explicitly:
  \begin{itemize}
    \item the index set of $\Psi_4$ is
      \begin{align*}
        I_4 &= \big\{ 1_\omega, 2_\omega \colon \omega \in \{0000, 1000, 0100, 1100, 0101, 0110, 0111\} \big\} \\
        &\cup \big\{ 3_\omega, 4_\omega \colon \omega \in \{0000, 1000, 1100\} \big\} \cup \big\{ 5_\omega, 6_\omega \colon \omega \in \{0101, 0110, 0111\} \big\} \, ;
      \end{align*}
    \item the space $\cV''$ is that associated to the above graph of vector spaces, and so is a subspace of ${V'}^{\cA}$ where $\cA = \{0000,1000,0100,1100,0101,0110,0111\}$;
    \item we have $W^{(4)}_{r_\omega} = \FF_p$ when $1 \le r \le 4$ or $\FF_p^2$ when $r = 5,6$; and
    \item $\psi^{(4)}_{r_\omega}((v'_\eta)_{\eta \in \cA}) = \psi_r(v'_\omega)$ for each $r_\omega \in I_4$.
  \end{itemize}
  This completes the first phase of the argument.

  We now perform our remaining $\MERGE$ operation.  This partitions all remaining forms apart from those in the $0000$ copy into two classes $A$ and $B$, by
  \begin{align*}
    \{ 1_{1000}, 2_{1000}, 1_{1100}, 2_{1100}, 5_{0110}, 6_{0110}, 5_{0111}, 6_{0111}, 5_{0101}, 1_{0101}, 2_{0100} \} &\mapsto A \\
    \{ 3_{1000}, 4_{1000}, 3_{1100}, 4_{1100}, 1_{0110}, 2_{0110}, 1_{0111}, 2_{0111}, 6_{0101}, 2_{0101}, 1_{0100} \} &\mapsto B
  \end{align*}
  and we call the merged datum $\Psi_5$.  This corresponds to the annotated diagram discussed above:
  
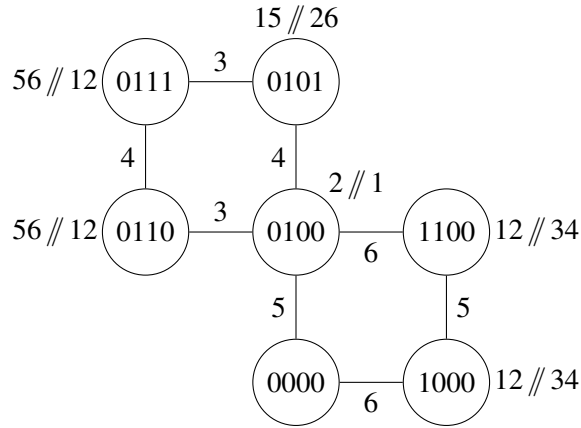
\begin{center}
    \begin{tikzpicture}
      \node[circle,draw] (0000) at (0,0) {$0000$};
      \node[circle,draw] (1000) at (2,0) {$1000$};
      \node[right] at ($(1000) + (0.5, 0)$) {$12 \sslash 34$};
      \node[circle,draw] (0100) at (0,2) {$0100$};
      \node[above right] at ($(0100) + (0.3, 0.3)$) {$2 \sslash 1$};
      \node[circle,draw] (1100) at (2,2) {$1100$};
      \node[right] at ($(1100) + (0.5, 0)$) {$12 \sslash 34$};
      \draw (0000) -- (1000) node[midway,below] {$6$};
      \draw (0100) -- (1100) node[midway,below] {$6$};
      \draw (0000) -- (0100) node[midway,left] {$5$};
      \draw (1000) -- (1100) node[midway,right] {$5$};

      \node[circle,draw] (0110) at (-2,2) {$0110$};
      \node[left] at ($(0110) + (-0.5, 0)$) {$56 \sslash 12$};

      \draw (0100) -- (0110) node[midway,above] {$3$};

      \node[circle,draw] (0101) at (0,4) {$0101$};
      \node[above] at ($(0101) + (0.0, 0.5)$) {$15 \sslash 26$};

      \node[circle,draw] (0111) at (-2,4) {$0111$};
      \node[left] at ($(0111) + (-0.5, 0)$) {$56 \sslash 12$};

      \draw (0101) -- (0111) node[midway,above] {$3$};

      \draw (0100) -- (0101) node[midway,left] {$4$};
      \draw (0110) -- (0111) node[midway,left] {$4$};
    \end{tikzpicture}
  \end{center}
  and we note the index set is now $\{r_{0000} \colon 1 \le r \le 4\} \cup \{A, B\}$.

  Finally, we wish to dominate $\Psi_5$ trivially by an augmented datum $\Psi'$.  Recall that $\Psi'$ is as follows:
  \begin{itemize}
    \item its base space is $V'$;
    \item the index set is $\{1,\dots,6\}$;
    \item the spaces $W'_i$ are given by $W'_i = \FF_p$ for $1 \le i \le 4$ and $W'_5, W'_6 = \FF_p^2$; and
    \item $\psi'_i = \psi_i$ for $1 \le i \le 4$, and $\psi'_i(v,t) = (\phi'_i(v), t + \chi'_i(v))$ for $i = 5,6$, where $\phi'_5, \phi'_6$ are the given forms satisfy $[\phi'_5] = X'_5$, $[\phi'_6] = X'_6$, and $\chi'_5, \chi'_6 \in \ell$ are forms we may choose.
  \end{itemize}
  In what follows we identify indices $r_{0000}$ and $r$ for $1 \le r \le 4$, $A$ and $5$, and $B$ and $6$.  Our remaining task is therefore to construct linear maps $\CMjmath \colon V' \to \cV''$, $\nu_5 \colon W'_5 \to W^{(5)}_5$ and $\nu_6 \colon W'_6 \to W^{(5)}_6$ which, together with the identity maps $W'_i \to W^{(5)}_i$ for $1 \le i \le 4$, satisfy the conditions of Proposition \ref{prop:trivial-dominate}.

  We fix some notation.  Again write $X_1, \dots, X_6$ for the points $[\phi_1],\dots,[\phi_6]$ in $\PP(V^\ast)$.  Let $Y$, $Z$, $X'_5$, $X'_6$ be defined as above (see Figure \ref{fig:main-construction}); that is, $Y$ is the intersection of the lines $X_1 X_5$, $Z$ is the intersection of the lines $X_2 X_6$ and $X_3 X_4$, $X'_5$ is $X_2 Y \cap \ell$, and $X'_6$ is $X_1 Z \cap \ell$.  Write $H_i = \ker(\psi_i)^\perp \subseteq V'^\ast$ for each $1 \le i \le 6$.
  
  Also recall $V' = V \oplus \FF_p$; so $V$ is naturally a subspace $\{(v,0) \colon v \in V\}$ of $V'$, and we write $T = \spn((0,1))$ for the other summand, so that $V' = V \oplus T$.  Dually, we may make an identification $V'^\ast = V^\ast \oplus T^\ast$, and thereby identify $V^\ast$ and $T^\ast$ with subspaces of $V'^\ast$.  Let $\xi \in {V'}^\ast$ be the linear form $\xi(v,t) = t$, meaning that $\spn(\xi) = T^\ast$.

  We make a simplifying observation.  If $Y = Z$, then $X'_6 = X_5$ and $X'_5 = X_6$, so the effect of the whole block move was just to swap $X_5$ and $X_6$.  In this case, the result is trivially satisfied by exchanging the indices $5$ and $6$ (and ignoring everything we've done up to this point).  Hence we can assume $Y \ne Z$ in what follows.

  We isolate a linear algebraic lemma which states concretely what is needed for this $\TRIVIAL$ step.
  \begin{lemma}
    \label{lem:horrible-la}
    There exist subspaces $H'_5$, $H'_6$ of $\ell + T^\ast$ \uppar{which is itself a subspace of $V'^\ast$}, and linear maps $\tau_1,\tau_2,\tau_3,\tau_4 \colon V' \to V'$, with the following properties:
    \begin{enumerate}[label=(\roman*)]
      \item composition with $\tau_1$ fixes $\psi_3$ and $\psi_4$ \uppar{i.e., $\psi_3 \circ \tau_1 = \psi_3$ and $\psi_4 \circ \tau_1 = \psi_4$};
      \item composition with $\tau_2$ fixes $\psi_5$ and $\psi_6$ \uppar{i.e., $\psi_5 \circ \tau_2 = \psi_5$ and $\psi_6 \circ \tau_2 = \psi_6$};
      \item similarly, $\psi_3 \circ \tau_3 = \psi_3$ and $\psi_6 \circ \tau_4 = \psi_6$;
      \item $\tau_2^\ast \tau_1^\ast (H_1 + H_5) \subseteq H'_5$ and $\tau_2^\ast \tau_1^\ast(H_2 + H_6) \subseteq H'_6$;
      \item $\tau_2^\ast(H_2) \subseteq H'_5$ and $\tau_2^\ast (H_1) \subseteq H'_6$;
      \item $\tau_2^\ast \tau_1^\ast \tau_3^\ast(H_5 + H_6) \subseteq H'_5$ and $\tau_2^\ast \tau_1^\ast \tau_3^\ast(H_1 + H_2) \subseteq H'_6$;
      \item $\tau_2^\ast \tau_4^\ast(H_1 + H_2) \subseteq H'_5$ and $\tau_2^\ast \tau_4^\ast(H_3+H_4) \subseteq H'_6$;
      \item $H'_5$ and $H'_6$ have dimension at most $2$, and $H'_5 \cap V^\ast$ and $H'_6 \cap V^\ast$ are contained in the $1$-dimensional subspaces corresponding to $X'_5$, $X'_6$ respectively.
    \end{enumerate}
  \end{lemma}

  Indeed, suppose this lemma holds.  We may define
  \begin{align*}
    \CMjmath \colon V' &\to \cV'' \\ 
    v' &\mapsto \begin{cases}
      0000 &\mapsto v' \\
      0100 &\mapsto \tau_2(v') \\
      1000 &\mapsto \tau_4 \tau_2(v') \\
      1100 &\mapsto \tau_4 \tau_2(v') \\
      0101 &\mapsto \tau_1 \tau_2(v') \\
      0110 &\mapsto \tau_3 \tau_1 \tau_2(v') \\
      0111 &\mapsto \tau_3 \tau_1 \tau_2(v')
    \end{cases}
  \end{align*}
  which is perhaps best summarized by further annotating the above diagram as follows:
  
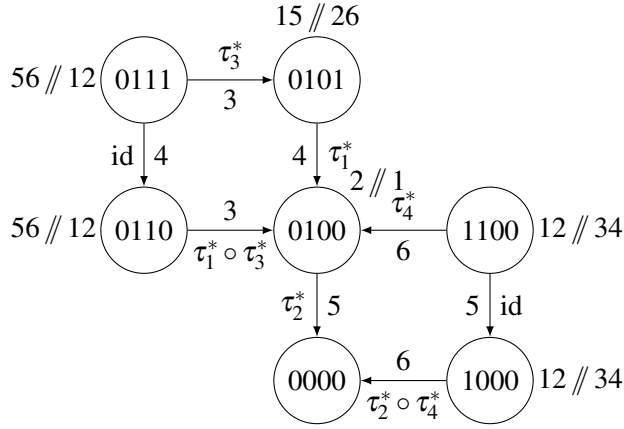
\begin{center}
    \begin{tikzpicture}
      \node[circle,draw] (0000) at (0,0) {$0000$};
      \node[circle,draw] (1000) at (2.3,0) {$1000$};
      \node[right] at ($(1000) + (0.5, 0)$) {$12 \sslash 34$};
      \node[circle,draw] (0100) at (0,2) {$0100$};
      \node[above right] at ($(0100) + (0.3, 0.3)$) {$2 \sslash 1$};
      \node[circle,draw] (1100) at (2.3,2) {$1100$};
      \node[right] at ($(1100) + (0.5, 0)$) {$12 \sslash 34$};
      \draw[latex-] (0000) -- (1000) node[midway,below] {$\tau_2^\ast \circ \tau_4^\ast$} node[midway,above] {$6$};
      \draw[latex-] (0100) -- (1100) node[midway,above] {$\tau_4^\ast$} node[midway,below] {$6$};
      \draw[latex-] (0000) -- (0100) node[midway,left] {$\tau_2^\ast$} node[midway,right] {$5$};
      \draw[latex-] (1000) -- (1100) node[midway,right] {$\id$} node[midway,left] {$5$};

      \node[circle,draw] (0110) at (-2.3,2) {$0110$};
      \node[left] at ($(0110) + (-0.5, 0)$) {$56 \sslash 12$};

      \draw[latex-] (0100) -- (0110) node[midway,below] {$\tau_1^\ast \circ \tau_3^\ast$} node[midway,above] {$3$};

      \node[circle,draw] (0101) at (0,4) {$0101$};
      \node[above] at ($(0101) + (0.0, 0.5)$) {$15 \sslash 26$};

      \node[circle,draw] (0111) at (-2.3,4) {$0111$};
      \node[left] at ($(0111) + (-0.5, 0)$) {$56 \sslash 12$};

      \draw[latex-] (0101) -- (0111) node[midway,above] {$\tau_3^\ast$} node[midway,below] {$3$};

      \draw[latex-] (0100) -- (0101) node[midway,right] {$\tau_1^\ast$} node[midway,left] {$4$};
      \draw[latex-] (0110) -- (0111) node[midway,left] {$\id$} node[midway,right] {$4$};
    \end{tikzpicture}
  \end{center}
  Statements (i)--(iii) ensure that $\CMjmath$ makes sense, i.e.~that all the compatibility conditions in the definition of $\cV''$ are satisfied.  By the fact that $H'_5, H'_6 \subseteq \ell + T^\ast$ and statement (viii), for any $\phi'_5, \phi'_6$ with $[\phi'_5] = X'_5$, $[\phi'_6] = X'_6$ we can find $\chi'_5, \chi'_6 \in \ell$ such that $\spn(\phi'_5, \xi + \chi'_5) \supseteq H'_5$ and $\spn(\phi'_6, \xi + \chi'_6) \supseteq H'_6$; we use these $\chi'_5$, $\chi'_6$ to complete the definition of $\psi'_5$, $\psi'_6$ and thereby $\Psi'$.  Then, statements (iv)--(vii) are precisely what we need to deduce that
  \begin{align*}
    \ker\left( \psi^{(5)}_A \circ \CMjmath \right) &\supseteq {H'_5}^\perp \supseteq \ker \psi'_5 \\
    \ker\left( \psi^{(5)}_B \circ \CMjmath \right) &\supseteq {H'_6}^\perp \supseteq \ker \psi'_6
  \end{align*}
  and as before this guarantees that there exist unique maps $\nu_5 \colon W'_5 \to W^{(5)}_A$ and $\nu_6 \colon W'_6 \to W^{(5)}_B$ such that $\psi^{(5)}_A \circ \CMjmath = \nu_5 \circ \psi'_5$ and $\psi^{(5)}_B \circ \CMjmath = \nu_6 \circ \psi'_6$ respectively.

  The proof of this lemma is unpleasant and technical linear algebra, and will occupy the rest of the section.
  
  \begin{proof}[Proof of Lemma \ref{lem:horrible-la}]
    Note that $H_1,\dots,H_4$ are $1$-dimensional subspaces of $V^\ast$ corresponding to $X_1,\dots,X_4$, i.e.~$H_i = \spn(\phi_i)$.  Also, $H_5, H_6$ are subspaces of $\ell + T^\ast$ of dimension $2$ such that $H_5 \cap V^\ast$ and $H_6 \cap V^\ast$ are the $1$-dimension subspaces $\spn(\phi_5)$, $\spn(\phi_6)$ corresponding to $X_5, X_6$ respectively.

    We first construct the map $\tau_1 \colon V' \to V'$. Roughly speaking, $\tau_1^\ast$ is a projection from $\PP(V'^\ast)$ to the line $X_3 X_4$, which collapses the line $X_1 X_5$ onto $Y$ and the line $X_2 X_6$ onto $Z$.  Specifically, we want the following.
    
    \begin{claim*}
      We can find $\tau_1 \colon V' \to V'$ satisfying the following conditions: $\tau_1^\ast|_{H_3+H_4}$ is the identity; and writing $H_Y = \tau_1^\ast(H_1 + H_5)$ and $H_Z = \tau_1^\ast(H_2 + H_6)$, then $H_Y$ is just the $1$-dimensional subspace corresponding to $Y$ and $H_Z$ the $1$-dimensional subspace corresponding to $Z$.
    \end{claim*}
    Note that the fact $\tau_1^\ast|_{H_3 + H_4} = \id$ implies (i) from the statement.
    \begin{proof}[Proof of Claim]
      Let $D$ be the intersection point of the lines $X_1 X_5$ and $X_2 X_6$ (which exists as, say, $X_1 X_2 X_5$ are not collinear).  Since we are assuming $Y \ne Z$, it follows that $D$ does not lie on the line $X_3 X_4$.

      Note $\dim (H_1 + H_5) = 3$ and $\dim (H_2 + H_6) = 3$, and that $(H_1 + H_5) \cap V^\ast$ and $(H_2 + H_6) \cap V^\ast$ are $2$-dimensional subspaces correspond to the lines $X_1 X_5$ and $X_2 X_6$ repsectively.  Writing $W = (H_1 + H_5) \cap (H_2 + H_6)$, we have $\dim W = 2$ (as $H_1+H_5 \ne H_2 + H_6$), and furthermore the intersection $W \cap V^\ast$ is precisely the $1$-dimensional subspace corresponding to $D = X_1 X_5 \cap X_2 X_6$.

      Hence, we may pick some non-zero $y \in W \cap V^\ast$, and extend this to a basis $x,y$ for $W$; so $[y] = D$ and necessarily $x \notin V^\ast$.

      It follows that $\phi_3, \phi_4, x, y$ is a basis for $V'^\ast$, and so we may define $\tau_1$ by:
      \begin{align*}
        \tau_1^\ast(\phi_3) &= \phi_3 \\
        \tau_1^\ast(\phi_4) &= \phi_4 \\
        \tau_1^\ast(x) &= 0 \\
        \tau_1^\ast(y) &= 0
      \end{align*}
      which immediately implies that $\tau_1^\ast|_{H_3+H_4}$ is the identity.

      Since $\dim (H_1 + H_5) = 3$ and $x,y \in H_1 + H_5$ are linearly independent vectors mapped to $0$ by $\tau_1^\ast$, it follows that $H_Y = \tau_1^\ast(H_1 + H_5)$ has dimension at most $1$.  Moreover, any vector in the $1$-dimension subspace corresponding to $Y$ is in $H_1 + H_5$ and also in $H_3 + H_4$, so is fixed by $\tau_1^\ast$.  It follows that $H_Y$ contains the $1$-dimensional subspace corrosponding to $Y$; so in fact $H_Y$ is exactly this subspace. 
      
      A parallel argument shows $H_Z$ is exactly the $1$-dimensional subspace $Z$.
    \end{proof}

    The construction of $\tau_2$ is similar, only in reverse, i.e.~projecting back onto the subspace $\ell + T^\ast$.
    \begin{claim*}
      We can find a map $\tau_2 \colon V' \to V'$ such that the following hold.  Define $H'_5 = \tau_2^\ast(H_Y + H_2)$ and $H'_6 = \tau_2^\ast(H_Z +H_1)$.  Then
        $H'_5, H'_6$ are subspaces of $\ell + T^\ast$ of dimension at most $2$; and
        $H'_5 \cap V^\ast$, $H'_6 \cap V^\ast$ are contained in the $1$-dimension subspace corresponding to $X'_5$, $X'_6$ respectively.
    \end{claim*}
    Note this gives the definition of the subspaces $H'_5$ and $H'_6$ from the statement. 
    \begin{proof}[Proof of claim]
      Let $e_1,e_2,e_3$ be any basis for $\ell + T^\ast$; so (say) $\phi_1, e_1, e_2, e_3$ is a basis for $V'^\ast$.  Define $\tau_2$ by:
      \begin{align*}
        \tau_2^\ast(e_i) &= e_i \ \ (i = 1,2,3) \,; \\
        \tau_2^\ast(\phi_1) &= \xi \, .
      \end{align*}
      So, $\tau_2^\ast$ is a projection onto the subspace $\ell + T^\ast$, and in particular its image is $\ell + T^\ast$.  Also, $H_1, H_2$, $H_Y, H_Z$ all have dimension $1$, and so the first part of the claim follows.

      Suppose $x \in H'_5 \cap V^\ast$.  Necessarily $x \in \ell$, and $x = \tau_2^\ast(y)$ for some $y \in H_Y + H_2$.  Since $y \in V^\ast$, we may write $y = \alpha \phi_1 + z$ for some $\alpha \in \FF_p$ and $z \in \ell$, and then $x = \alpha \xi + z$.  But $x \in \ell$, $z \in \ell$ and $\xi \notin \ell$, so $\alpha = 0$ and $x = y = z$.  We deduce that $x$ lies in $H_Y + H_2$ (as $y$ does) and in $\ell$; so $x$ lies in the $1$-dimension subspace $(H_Y + H_2) \cap \ell$ which corresponds precisely to $X'_5$.

      A parallel argument shows that $H'_6 \cap V^\ast$ is contained in the subspace corresponding to $X'_6$.
    \end{proof}
    At this point properties (ii), (iv), (v) and (viii) from the statement are satisfied, in addition to (i) as discussed. Indeed, (ii) follows from the requirement $\tau_2^\ast|_{\ell + T^\ast} = \id$, since $(\ker \psi_5)^\perp, (\ker \psi_6)^\perp \subseteq \ell + T^\ast$.  The facts $\tau_1^\ast(H_1+H_5) = H_Y$ and $\tau_1^\ast(H_Y) \subseteq H'_5$, and similarly $\tau_2^\ast(H_2+H_6) = H_Z$ and $\tau_2^\ast(H_Z) \subseteq H'_6$, give (iv).  Property (v) is immediate from the definition of $H'_5$, $H'_6$, and (viii) is contained in the previous claim.

    Our final task is to construct $\tau_3$ and $\tau_4$.  Specifically, we want:
    \begin{claim*}
      There exist maps $\tau_3 \colon V' \to V'$ and $\tau_4 \colon V' \to V'$ such that
      \begin{itemize}
        \item $\psi_3 \circ \tau_3 = \psi_3$;
        \item $\tau_3^\ast(H_5 + H_6) \subseteq H_1 + H_5$;
        \item $\tau_3^\ast(H_1 + H_2) \subseteq H_2 + H_6$;
      \end{itemize}
      and
      \begin{itemize}
        \item $\psi_6 \circ \tau_4 = \psi_6$;
        \item $\tau_4^\ast(H_1 + H_2) \subseteq H_Y + H_2$;
        \item $\tau_4^\ast(H_3 + H_4) \subseteq H_Z + H_1$.
      \end{itemize}
    \end{claim*}

  Combined with the properties of $\tau_1$ and $\tau_2$ we have already shown, this suffices for the remaining parts (iii),(vi),(vii) of the statement.

  We isolate yet another linear algebra sub-claim.
  \begin{lemma}
    Given a tuple $(w,U_1, U_2)$ where $U_1,U_2$ are two distinct $2$-dimensional subspaces of a vector space $W$ of dimension $3$, and $0 \ne w \in W$ is a point not in either subspace; and another such configuration $(w', U'_1, U'_2)$; there is some isomorphism $\theta \colon W \to W$ mapping $w \mapsto w'$, $U_1 \mapsto U'_1$, $U_2 \mapsto U'_2$.
  \end{lemma}
  \begin{proof}
    After a change of coordinates, we may assume $U_1 = \{(x,y,z) \in \FF_p^3 \colon x = 0\}$ and $U_2 = \{(x,y,z) \in \FF_p^3 \colon y = 0\}$ (e.g.~by considering two corresponding points in $W^\ast$ and extending to a basis).  Suppose $w = (\alpha, \beta, \gamma)$ in these coordinates; so $\alpha, \beta \ne 0$ by assumption.  By a further rescaling of coordinates we may therefore assume $\alpha = \beta = 1$.  Finally, a change of variables $z' = z - \gamma x$ means $w = (1, 1, 0)$ and $U_1, U_2$ are unchanged.

    Repeating this argument for $(w', U'_1, U'_2)$ and interpreting the changes of coordinates as an isomorphism gives the result.
  \end{proof}

  \begin{proof}[Proof of claim]
    Let $\theta_3 \colon V^\ast \to V^\ast$ be the map given by the lemma applied to the tuples $(\phi_3, \ell, H_1 + H_2)$ and $(\phi_3, (H_1 + H_5) \cap V^\ast, (H_2 + H_6) \cap V^\ast)$ in $V^\ast$.  Note that $X_3$ is not on $\ell$, $X_1 X_2$, $X_1 X_5$ or $X_2 X_6$ by our assumptions, and $X_1 X_2 \ne \ell$, $X_1 X_5 \ne X_2 X_6$, so the hypotheses are satisfied.  Then let $\tau_3^\ast(v, t) = (\theta_3(v), 0)$.  It follows that this has the properties claimed.

    Now let $\theta_4 \colon V^\ast \to V^\ast$ be the map given by the lemma applied to $(\phi_6, H_1 + H_2, H_3 + H_4)$ and $(\phi_6, H_Y + H_2, H_Z + H_1)$.  Again, $X_6$ does not lie on $X_1 X_2$ or $X_3 X_4$,  or on $Y X_2$ or $Z X_1$ (as the latter two imply respectively that $Y=Z$ or $X_1 X_2 X_6$ are collinear); and $X_1 X_2 \ne X_3 X_4$ and $Y X_2 \ne Z X_1$ (for many reasons); so this is valid.
    
    Now let $\tau_4^\ast$ be the unique map defined by $\tau_4^\ast(v) = \theta_4(v)$ for any $v \in V^\ast$ and $\tau_4^\ast(\chi_6, 1) = (\chi_6, 1)$, where $\chi_6$ is the linear form from the definition of $\Psi$, as in Definition \ref{def:augmented}.  (This is possible: choose a basis for $V^\ast$ and extend it to a basis for $V'^\ast$ by adding $(\chi_6, 1)$; then define $\tau_4^\ast$ suitably on this basis.)  Again, this gives the desired properties.
  \end{proof}

  This concludes the proof of Lemma \ref{lem:horrible-la}, and thereby that of Lemma \ref{lem:key-lemma}.
  \end{proof}
  \renewcommand{\qedsymbol}{}
\end{proof}

This is the last ingredient in the proof of Theorem \ref{thm:positive}.  We briefly summarize the proof as a whole, as the different parts have been spread over the last few sections.

First one calculates the point $[r:s] \in \PP^1(\FF_p)$, given explicitly in Lemma \ref{lem:find-alpha}, corresponding to the point $X_5$ on $\ell$ in our chosen coordinates.

Next, we convert the standard datum given into an augmented datum (by Lemma \ref{lem:make-augmented}).

In the main part of the argument, we apply Lemma \ref{lem:key-lemma} repeatedly, under various permutations of $\{1,\dots,4\}$, following the steps from from Lemma \ref{lem:euclid} applied to the point $[r:s]$.

If at any point we arrive at a datum where $X_i X_j X_k$ are collinear for some $1 \le i < j \le 4$ and $k = 5, 6$, we terminate this process early; but if it runs to completion, some such collinearity is guaranted at the end.  By Lemma \ref{lem:make-augmented} again (and Lemma \ref{lem:no-final-coincide}) we dominate this by the corresponding standard datum.

Finally, we apply Proposition \ref{prop:skew-collinear}, or the standard Cauchy--Schwarz complexity bound (Proposition \ref{prop:csc}), to control this final datum by $\|f_1\|_{U^2}^{1/2}$ or $\|f_1\|_{U^2}$ respectively.  By keeping track of the various domination statements, and noting in particular that we did not apply Lemma \ref{lem:key-lemma} too many times, we deduce the required bound on the original linear datum.

\section{A proof of Theorem \ref{thm:negative}}
\label{sec:negative}

Here we describe the construction of the counterexample described in Theorem \ref{thm:negative}. 

As in the statement, let $p \equiv \pm 1 \pmod{8}$ be a large prime.  The congruence condition ensures that $2$ is a quadratic residue modulo $p$.  In what follows, we will assume that some choice of square root of $2$ in $\FF_p$ has been fixed, and refer to it simply as $\sqrt{2}$.

We let $X \subseteq \FF_p$ denote the two-dimensional arithmetic progression:
\[
  X = \left\{ a + b \sqrt{2} \colon a,b \in \ZZ,\, |a|,|b| \le \alpha p^{1/2} \right\}
\]
for some small absolute constant $\alpha > 0$ to be specified.

We note that any value $x \in \FF_p$ has at most one representation as $x = a + b \sqrt{2}$ where $a,b$ are integers with $|a|,|b| \le p^{1/2} / 4$.  Indeed, if $x = a + b \sqrt{2} = a'+b' \sqrt{2}$ then $(a-a')^2 - 2(b-b')^2$ is a multiple of $p$; but it has absolute value at most $\max((a-a')^2, 2(b-b')^2) \le 2 (p^{1/2} /2 )^2 < p$.  Hence, $(a-a')^2 = 2 (b-b')^2$ which is a contradiction unless $a=a'$, $b=b'$.

Now, define the system $\Phi$ of linear forms $\phi_1,\dots,\phi_6 \colon \FF_p^3 \to \FF_p$ by:
\begin{align*}
  \phi_1(x,y,z) &= (1+\sqrt{2})x + z & \phi_3(x,y,z) &= (1 + \sqrt{2})y + z &  \phi_5(x,y,z) &= x+y+\sqrt{2} z \\
  \phi_2(x,y,z) &= (1-\sqrt{2})x + z & \phi_4(x,y,z) &= (1 - \sqrt{2})y + z &  \phi_6(x,y,z) &= x+y-\sqrt{2} z \, .
\end{align*}
We claim that these forms do not lie on a conic, i.e.~the system has true complexity $1$.  Since $p \ne 2$, we need not distinguish between symmetric bilinear forms and quadratic forms, so it makes sense to write
\begin{align*}
  \phi_1^2(x,y,z) &= (3+2\sqrt{2}) x^2 + z^2 + 2(1+\sqrt{2}) xz \\
  \phi_2^2(x,y,z) &= (3-2\sqrt{2}) x^2 + z^2 + 2(1-\sqrt{2}) xz \\
  \phi_3^2(x,y,z) &= (3+2\sqrt{2}) y^2 + z^2 + 2(1+\sqrt{2}) yz \\
  \phi_4^2(x,y,z) &= (3-2\sqrt{2}) y^2 + z^2 + 2(1-\sqrt{2}) yz \\
  \phi_5^2(x,y,z) &= x^2 + y^2 + 2 z^2 + 2 xy + 2\sqrt{2} xz + 2 \sqrt{2} yz \\
  \phi_6^2(x,y,z) &= x^2 + y^2 + 2 z^2 + 2 xy - 2\sqrt{2} xz - 2 \sqrt{2} yz
\end{align*}
and then it suffices to verify that the matrix (whose columns correspond to $x^2,y^2,z^2,xy,xz,yz$)
\[
  \begin{pmatrix}
    3 + 2\sqrt{2} & 0 & 1 & 0 & 2 + 2\sqrt{2} & 0 \\
    3 - 2\sqrt{2} & 0 & 1 & 0 & 2 - 2\sqrt{2} & 0 \\
    0 & 3 + 2\sqrt{2} & 1 & 0 & 0 & 2 + 2\sqrt{2} \\
    0 & 3 - 2\sqrt{2} & 1 & 0 & 0 & 2 - 2\sqrt{2} \\
    1 & 1 & 2 & 2 & 2 \sqrt{2} & 2 \sqrt{2} \\
    1 & 1 & 2 & 2 & -2 \sqrt{2} & -2 \sqrt{2} 
  \end{pmatrix}
\]
is non-singular.  In fact, the determinant of this matrix is $512 \sqrt{2}$, and so $\phi_1,\dots,\phi_6$ do not lie on a conic.

Our reason for choosing these specific forms is that they nonetheless satisfy a kind of ``skew conic'' identity which we now describe.  Suppose we define forms $\widetilde{\phi_1},\dots,\widetilde{\phi_6} \colon \ZZ[\sqrt{2}]^3 \to \ZZ[\sqrt{2}]$ using the same coefficients as above, now thought of as elements of the ring $\ZZ[\sqrt{2}]$.  These also induce forms $\QQ(\sqrt{2})^3 \to \QQ(\sqrt{2})$ in the obvious way.  Write $\sigma(a+b\sqrt{2}) = a - b\sqrt{2}$ for the Galois conjugate and $N(z) = z \sigma(z)$ for the norm of $z \in \QQ(\sqrt{2})$ (so $N(a+b\sqrt{2}) = a^2-2 b^2$). Then the ``skew conic'' identity is the fact that:
\begin{equation}
  \label{eq:norm-relation}
\begin{aligned}
  N(\widetilde{\phi_1}&(x,y,z)) - N(\widetilde{\phi_2}(x,y,z)) + N(\widetilde{\phi_3}(x,y,z)) - N(\widetilde{\phi_4}(x,y,z)) + N(\widetilde{\phi_5}(x,y,z)) - N(\widetilde{\phi_6}(x,y,z)) \\
  &= \begin{aligned}[t]
    & ((1+\sqrt{2}) x + z)((1 - \sqrt{2})\sigma(x) + \sigma(z)) - ((1-\sqrt{2}) x + z)((1 + \sqrt{2})\sigma(x) + \sigma(z)) \\
    + & ((1+\sqrt{2}) y + z)((1 - \sqrt{2})\sigma(y) + \sigma(z)) - ((1-\sqrt{2}) y + z)((1 + \sqrt{2})\sigma(y) + \sigma(z)) \\
    + & (x + y + \sqrt{2}z)(\sigma(x) + \sigma(y) - \sqrt{2} \sigma(z)) - (x + y - \sqrt{2}z)(\sigma(x) + \sigma(y) + \sqrt{2} \sigma(z))
  \end{aligned} \\
  &= 2\sqrt{2} (x \sigma(z) - \sigma(x) z) + 2 \sqrt{2} (y \sigma(z) - \sigma(y) z) + 2 \sqrt{2} ((\sigma(x) + \sigma(y))z - (x+y) \sigma(z)) \\
  &= 0
\end{aligned}
\end{equation}
for any $x,y,z \in \QQ(\sqrt{2})$.

So, we can define a function
\begin{align*}
  f \colon \FF_p &\to \CC \\
  x &\mapsto \begin{cases} 
    e\big(R\, N(a+b\sqrt{2}) \big) &\colon x \in X,\, x=a+b\sqrt{2} \text{ for } a,b \in \ZZ, |a|,|b| \le \alpha p^{-1/2} \\
      0 &\colon \text{ otherwise }
  \end{cases}
\end{align*}
for some parameter $R \in \RR/\ZZ$ to be specified; so $f$ is supported on $X$.  We claim, for say $\alpha = 1/49$, that for every $x,y,z \in \FF_p$ the expression
\[
  f(\phi_1(x,y,z)) \overline{f(\phi_2(x,y,z))}
  f(\phi_3(x,y,z)) \overline{f(\phi_4(x,y,z))}
  f(\phi_5(x,y,z)) \overline{f(\phi_6(x,y,z))}
\]
takes value either $1$, if all six forms take values in $X$, or $0$ otherwise.  This is essentially a kind of Fre\u{\i}man isomorphism argument over $\ZZ[\sqrt{2}]$.

Indeed, note that tuples $(\phi_1(x,y,z),\dots,\phi_6(x,y,z))$ or $\big(\widetilde{\phi_1}(x,y,z),\dots,\widetilde{\phi_6}(x,y,z)\big)$ for $x,y,z \in \FF_p$ or $x,y,z \in \QQ(\sqrt{2})$ respectively are precisely those tuples $(r_1,\dots,r_6)$ satisfying the equations
\begin{align*}
  (1 - \sqrt{2})(r_1 - r_3) - (1 + \sqrt{2})(r_3 - r_4) &= 0\\
  2 r_1 + 2 r_3 - (1 + 2 \sqrt{2}) r_5 - r_6 &= 0 \\
  2 r_2 + 2 r_4 - r_5 - (1 - 2 \sqrt{2}) r_6 &= 0\ .
\end{align*}
So, if $\phi_1(x,y,z),\dots,\phi_6(x,y,z)$ are all in $X$, write $\phi_i(x,y,z) = a_i + b_i \sqrt{2}$ for the unique integers $a_i, b_i$ with $|a_i|,|b_i| \le \alpha p^{-1/2}$, and let $r_i = a_i + b_i\sqrt{2}$ be the corresponding elements of $\ZZ[\sqrt{2}]$.  Then each of the left hand sides of the equations above has the form $u + v \sqrt{2}$ for integers $u, v$ with $|u|,|v| \le 12 \alpha p^{1/2} < p^{1/2} / 4$, and is congruent to $0 = 0 + 0 \sqrt{2}$ modulo $p$, so by our earlier uniqueness argument must be $0$.  Hence $(r_1,\dots,r_6) = \big(\widetilde{\phi_1}(x,y,z),\dots,\widetilde{\phi_6}(x,y,z)\big)$ for some $x,y,z \in \QQ(\sqrt{2})$, so \eqref{eq:norm-relation} applies and the claim follows.

Given that $\phi_1(x,y,z),\dots,\phi_6(x,y,z) \in X$ for any $x,y,z$ of the form $a+b\sqrt{2}$ for integers $a,b$ with $|a|,|b| \le p^{1/2} / 199$, we deduce that
\[
  \left|\Lambda_{\Phi}(f,\overline{f},f,\overline{f},f,\overline{f}) \right| \ge 10^{-12}
\]
for $p$ sufficiently large, as required.

Finally we need to consider $\|f\|_{U^2}$.  This is a fairly standard estimate on quadratic exponential sums, but with some variations.  For simplicity we use a mean value strategy.

For $x \in X$, let $N(x)$ denote $N(a+b\sqrt{2}) = a^2 - 2 b^2$ where $a,b$ are the unique integers in $[-\alpha p^{1/2}, \alpha p^{1/2}]$ with $x = a + b \sqrt{2}$.  Then for any $x,h,h' \in \FF_p$ such that $x, x+h, x+h', x+h+h'$ are in $X$, we have $h = r + s \sqrt{2}$, $h' = r'+s'\sqrt{2}$ for some unique integers $r,s,r',s' \in [-2\alpha p^{1/2},2\alpha p^{1/2}]$, and
\[
  N(x) - N(x+h) - N(x+h') + N(x+h+h') = 2 r r' - 4 s s' \, .
\]
So, for any such $x,h,h'$, if we now take an average in the parameter $R$, we get
\begin{align*}
  \int_{R \in \RR/\ZZ} f(x) \overline{f(x+h)} \overline{f(x+h')} f(x+h+h') &= \int_{R \in \RR/\ZZ} e(R (2 r r' - 4 s s')) \\
  &= [r r' = 2 s s'] \, .
\end{align*}
It follows that
\[
  \int_{R \in \RR/\ZZ} \|f\|_{U^2}^4 = p^{-3} \sum_{\substack{x \in X \\ |r|,|s|,|r'|,|s'| \le 2 \alpha p^{1/2}}} 1_X(x) 1_X(x+h) 1_X(x+h') 1_X(x+h+h') [ r r' = 2 s s']
\]
and we note that for $r,r'$ fixed and any fixed $s' \ne 0$ there is at most one solution in $s$ to $r r' = 2 s s'$, so the right hand side is bounded by
\[
  2 p^{-3} |X| (1 + 4 \alpha p^{1/2})^3 \le p^{-1/2}
\]
whenever $p$ is sufficiently large.  Picking some $R$ for which $\|f\|_{U^2}^4$ is at most its mean value, the result follows.

%\newpage %% AUTHOR: please comment out this line.  It serves only
%%   to demonstrate both types of header line in daj-template.pdf

%%% AUTHOR: optional appendix here
%\appendix %% you may comment this out if no Appendix
%\section*{Appendix}
%\section{Improving the constants}
%Material is placed here as needed.

%%% AUTHOR: optional acknowledgments here
%\section*{Acknowledgments} %%  you may comment this out if no Ackno
%The authors are grateful to the anonymous reviewers for finding
%a bug in the main result.

%%% AUTHOR:
%%% Bibliography goes here. Note that the arXiv cannot process bibtex
%%% or biber bibliographies.  Example of acceptable bibliograpy format:

\bibliographystyle{amsplain}
\providecommand{\bysame}{\leavevmode\hbox to3em{\hrulefill}\thinspace}
\providecommand{\MR}{\relax\ifhmode\unskip\space\fi MR }
\providecommand{\MRhref}[2]{%
  \href{http://www.ams.org/mathscinet-getitem?mr=#1}{#2}
}
\providecommand{\href}[2]{#2}

%% AUTHOR: You can generate such a bibliography from a .bib file by 
%% running pdflatex/bibtex/pdflatex/pdflatex and then pasting the .bbl file
%% between \begin{thebibliography} and \end{bibliography}

%%% AUTHOR: Include a short description of each author following the
%%% structure below. Use the same short tags used previously.  
%%% Use \imageat{} and \imagedot{} instead of "@" and "." in
%%% email addresses-this replaces the symbols with graphics to avoid 
%%% e-mail address harvesting from the .pdf file
\begin{dajauthors}
\begin{authorinfo}[f]
  Freddie Manners \\
  Department of Mathematics \\
  Stanford University \\
  {\tt fmanners\imageat{}stanford.edu}
\end{authorinfo}
\end{dajauthors}

\end{document}